\documentclass[a4paper,11pt,reqno]{amsart}

\usepackage{amsmath,amsfonts,amsthm,amssymb,ascmac,amscd}
\usepackage{bm}
\usepackage{graphicx}
\usepackage{enumerate}
\usepackage{overpic}
\usepackage{mathtools}
\usepackage{hyperref} 
\usepackage{xcolor}
\hypersetup{
bookmarksnumbered=true, 
colorlinks=true, 
citecolor=black,
linkcolor=black,
urlcolor=black,
}
\numberwithin{equation}{section}
\usepackage{cleveref}
\crefname{def}{Definition}{Definitions}
\crefname{theorem}{Theorem}{Theorems}
\crefname{lemma}{Lemma}{Lemmas}
\crefname{corollary}{Corollary}{Corollaries}
\crefname{proposition}{Proposition}{Propositions}
\crefname{claim}{Claim}{Claims}
\crefname{example}{Example}{Examples}
\crefname{remark}{Remark}{Remarks}
\crefname{question}{Question}{Questions}
\crefname{equation}{}{}
\crefname{figure}{Figure}{Figures}
\theoremstyle{definition}
\newtheorem{theorem}{Theorem}[section]
\newtheorem*{theorem*}{Theorem}
\newtheorem{definition}[theorem]{Definition}
\newtheorem*{definition*}{Definition}
\newtheorem{proposition}[theorem]{Proposition}
\newtheorem*{proposition*}{Proposition}

\newtheorem*{example*}{Example}
\newtheorem{remark}[theorem]{Remark}
\newtheorem*{remark*}{Remark}

\newtheorem*{recall*}{Recall}
\newtheorem{lemma}[theorem]{Lemma}
\newtheorem*{lemma*}{Lemma}
\newtheorem{corollary}[theorem]{Corollary}
\newtheorem*{corollary*}{Corollary}

\newtheorem*{question*}{Question}

\newtheorem*{conjecture*}{Conjecture}

\newtheorem*{exercise*}{Exercise}

\newtheorem*{claim*}{Claim}

\newtheorem*{fact*}{Fact}
\newtheorem{theorema}{Theorem}

\newtheorem{corollarya}[theorema]{Corollary}
\newcommand{\Z}{\mathbb{Z}}
\newcommand{\R}{\mathbb{R}}
\newcommand{\C}{\mathbb{C}}

\DeclareMathOperator{\id}{id}

\let\Im\relax
\DeclareMathOperator{\Im}{Im}
\DeclareMathOperator{\pr}{pr}

\DeclareMathOperator{\ord}{ord}
\newcommand{\Teich}{\mathcal{T}}
\newcommand{\QD}{\mathrm{QD}}

\DeclareMathOperator{\Res}{Res}

\allowdisplaybreaks[4]
\title[Degeneration of hyperbolic surfaces with boundary]{Uniform Degeneration of hyperbolic surfaces with boundary along harmonic map rays}
\author{Kento Sakai}
\address{Graduate~School~of~Science, Osaka~University}
\email{u741819k@ecs.osaka-u.ac.jp}
\date{\today}
\begin{document}

\begin{abstract}
    We study the degeneration of hyperbolic surfaces along a ray given by the harmonic map parametrization of Teichmüller space.
    The direction of the ray is determined by a holomorphic quadratic differential on a punctured Riemann surface, which has poles of order $\geq 2$ at each puncture.
    We show that the rescaled distance functions of the universal covers of hyperbolic surfaces uniformly converge, on a certain non-compact region containing a fundamental domain, to the intersection number with the vertical measured foliation given by the holomorphic quadratic differential determining the direction of the ray.
    
    This implies that hyperbolic surfaces along the ray converge to the dual $\R$-tree of the vertical measured foliation in the sense of Gromov-Hausdorff.
    As an application, we determine the limit of the hyperbolic surfaces in the Thurston boundary.
\end{abstract}
\maketitle

\section{Introduction}

Let $S$ be a connected and oriented surface.
We suppose that $S$ admits a complete, finite area hyperbolic structure with geodesic boundary.
The \textit{Teichmüller space} of $S$ is the space of marked hyperbolic structures on $S$.
There are many kinds of parametrizations for the Teichmüller space (for example, Fenchel-Nielsen coordinates, Teichmüller's parametrization, the earthquake parametrization, the Fricke coordinate, and so on), and each of them demonstrates geometric properties of Teichmüller spaces.
In this paper, we are interested in the parametrization by the space of holomorphic quadratic differentials via harmonic maps between hyperbolic surfaces.

We here review the history of the parametrization given by harmonic maps.
If the surface is closed, 
a harmonic diffeomorphism uniquely exists in the homotopy class of a homeomorphism \cite{eells1964harmonic,hartman1967homotopic,schoen1978harmonic,sampson1978harmonic}.
Based on this result, Wolf \cite{wolf1989harmonic} proved that the Teichmüller space is homeomorphic to the vector space of holomorphic quadratic differentials on a Riemann surface by taking the associated Hopf differentials.
Wolf's parametrization of Teichmüller space was also found independently and by different means by Hitchin \cite{hitchin1987selfduality} and Wan \cite{wan1992constantmeancurvature}.
For hyperbolic surfaces with cusps, Lohkamp \cite{lohkamp1991harmonic} established an analogous parametrization by the space of holomorphic quadratic differentials with poles of order at least $1$.
For surfaces with closed boundaries, the parametrization is provided with the space of holomorphic quadratic differentials with poles of order 2 and fixed residues.
Gupta \cite{gupta2021harmonic} similarly parametrized the Teichmüller space for hyperbolic surfaces with crown ends using the space of holomorphic quadratic differentials with poles of order at least 3 and fixed principal parts.
Recently, Allegretti \cite{allegretti2021stability} addressed surfaces with punctures and boundaries of both types, and some exceptional types of surfaces, which had never been dealt with.

\begin{figure}[t]
   \centering
   \begin{overpic}[scale=0.6]{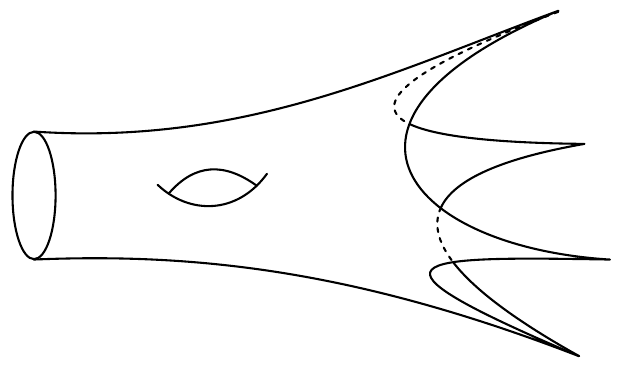}
   \end{overpic}
   \caption{Hyperbolic surface with a closed boundary (left side) and a crown end (right side).}
   \label{figure:HyperbolicSurface}
\end{figure}

In the case that a hyperbolic surface is closed or with cusps,
the interesting fact of this parametrization is that, if we deform a hyperbolic surface along a ray from the origin in the vector space of holomorphic quadratic differentials on a Riemann surface, these hyperbolic surfaces converge to a point in the Thurston boundary \cite{wolf1989harmonic,sakai2023compactification}.
Moreover, the limit point is the projective class of the vertical measured foliation determining the direction of the ray.
In terms of metric geometry, Wolf studied the degenerations of hyperbolic structures along a ray, by realizing them as metric spaces via harmonic maps.
As he mentioned in \cite{wolf1995rtree}, the distance functions of the metric lifted to the universal covering of the hyperbolic surfaces uniformly converge on compact sets.
Wolf also described degenerations of hyperbolic structures by using the notion of the equivariant Gromov-Hausdorff convergence, which was introduced by Paulin \cite{paulin1988rtree}. From these results, Wolf recovered the result of Morgan-Shalen \cite{morgan1984degeneration} in the case of $2$-dimensional hyperbolic manifolds.

In this paper, we consider the degeneration of hyperbolic surfaces with closed boundaries or crown ends (see \Cref{figure:HyperbolicSurface}). 
For the sake of simplicity, first we explain our result in the case that  $S$ is an ideal $k$-gon for $k\geq 3$.
In this case, we deal with harmonic maps from the complex plane $\C\cong \C P^1\setminus\{\infty\}$ to the hyperbolic plane $\mathbb{H}^2$ with polynomial Hopf differentials.
Harmonic maps of this type were first studied in \cite{han1995harmonic} where it was shown that their images form ideal polygons.
Let $\varphi$ be a polynomial quadratic differential of degree $k-2$ on $\C$, which has a pole of order $k+2$ at $\infty \in \C P^1$.
Then, there exists a unique marked hyperbolic ideal $k$-gon $[C_\varphi,f_\varphi] \in \Teich(S)$ satisfying the following condition: there exists a harmonic diffeomorphism $h_\varphi\colon \C\to C_\varphi$ such that the \textit{Hopf differential} of $h_\varphi$ is equal to $\varphi$,
where the Hopf differential of $h_\varphi$ is the $(2,0)$-part of the pullback metric ${h_\varphi}^\ast C_\varphi$.
Therefore, for each $t>0$, we obtain the hyperbolic surface $[C_{t\varphi},f_{t\varphi}]$ and the harmonic map $h_{t\varphi}$ corresponding to $t\varphi$.
The one-parameter family $\{[C_{t\varphi},f_{t\varphi}]\}_{t>0}$ of the hyperbolic surfaces is called the \textit{harmonic map ray} in the direction of $\varphi$.
The pullback metric ${h_{t\varphi}}^\ast C_{t\varphi}$ rescaled by $\frac1{4t}$ defines a distance function on $\C$.
For each $t>0$, let $d_t$ denote the distance function.
On the other hand, another metric space is obtained from $\varphi$.
Let $T(\varphi)$ denote the leaf space of the vertical measured foliation $\mathcal{F}_{\varphi,\mathrm{vert}}$.
The transverse measure of $\mathcal{F}_{\varphi,\mathrm{vert}}$ defines a distance $d_{T(\varphi)}$ on $T(\varphi)$.
The metric space $(T(\varphi),d_{T(\varphi)})$ is called the \textit{dual $\R$-tree} to $\mathcal{F}_{\varphi,\mathrm{vert}}$.
If $\varphi$ is a polynomial quadratic differential on $\C$, it is a simplicial metric tree.
Our main result in the case of an ideal $k$-gon is the following.
\begin{theorema}[\Cref{corollary:GromovHausdorff}]
    The rescaled metric spaces $(\C,d_t)$ uniformly converge to the simplicial metric tree $(T(\varphi),d_{T(\varphi)})$ in the sense of Gromov-Hausdorff as $t$ tends to $\infty$.
\end{theorema}
This theorem is obtained by showing that, for every $\varepsilon>0$ and every sufficiently large $t>0$, the projection map $\pr_\varphi\colon \C\to T(\varphi)$ to the leaf space defines an \textit{$\varepsilon$-approximation} between $(\C,d_t)$ and $(T(\varphi),d_{T(\varphi)})$, namely, for every $x_1,x_2\in \C$, $|d_t(x_1,x_2)-d_{T(\varphi)}(\pr_\varphi x_1,\pr_\varphi x_2)|<\varepsilon$.
In other words, the distance functions $d_t$ uniformly converge to $d_{T(\varphi)}\circ \operatorname{Pr}_\varphi$ on the complex plane $\C$, which is a non-compact domain.
Here, we define the map $\operatorname{Pr}_\varphi$ by $\operatorname{Pr}_\varphi(x_1,x_2)=(\pr_\varphi x_1,\pr_\varphi x_2)$.

Next, we present the general case.
Let $X$ be a closed Riemann surface of genus $g$ with $n$ punctures and let $\varphi$ be a holomorphic quadratic differential on $X$ with poles only at the punctures.
We suppose that the order of $\varphi$ at each puncture is at least $2$.
The orders of $\varphi$ at punctures determine a topological surface $S$ with closed boundaries or crown ends. 
Then, there exists a unique hyperbolic structure $[C_\varphi,f_\varphi]\in\mathcal{T}(S)$ satisfying the same condition as the ideal polygon case. Thus, having the hyperbolic surface $[C_{t\varphi},f_{t\varphi}]$ associated with $t\varphi$ for each $t>0$, we obtain the harmonic map ray $\{[C_{t\varphi},f_{t\varphi}]\}$.

To describe the limit of a harmonic map ray, we take the universal covering $\pi \colon \widetilde{X}\to X$.
Since $\widetilde{X}$ is conformally equivalent to the hyperbolic plane $\mathbb{H}$ (or the Euclidean plane $\mathbb{E}$), we identify $\widetilde{X}$ with $\mathbb{H}$ (or $\mathbb{E}$).
Let $\Gamma$ be the discrete free subgroup of the orientation-preserving isometry group $\operatorname{Isom}^+(\mathbb{H})$ of the hyperbolic plane (or $\operatorname{Isom}^+(\mathbb{E})$ of the Euclidean plane) with $X=\widetilde{X}/\Gamma$.
Taking the lift of ${h_{t\varphi}}^\ast C_{t\varphi}$ to $\widetilde{X}$ and rescaling the metric by $\frac1{4t}$, we get the distance function $d_t\colon\widetilde{X}\times\widetilde{X}\to \R_{\geq 0} $ from the rescaled metric on $\widetilde{X}$, which is invariant under the action of $\Gamma$.
To state our main theorem, we consider the hyperbolic plane (or the Euclidean plane) with the Gromov boundary $\partial_\infty \Gamma$ of $\Gamma$ attached in an equivariant manner.
We denote this union with a natural topology by $\widetilde{X}\cup \partial_\infty \Gamma$.
Its formal definition is introduced in \Cref{section:HyperbolicPlaneWithTheEndsOfTheCayleyGraph}.
Let $T(\varphi)$ be the dual $\R$-tree to the vertical measured foliation $\mathcal{F}_{\widetilde{\varphi},\mathrm{vert}}$ of $\widetilde{\varphi}$.
The transverse measure of $\mathcal{F}_{\widetilde{\varphi},\mathrm{vert}}$ defines a distance on $T(\varphi)$.
We define a map $\pr_\varphi\colon \widetilde{X}\to T(\varphi)$ as the projection to the leaf space of $\mathcal{F}_{\widetilde{\varphi},\mathrm{vert}}$.
\begin{theorema}[\Cref{theorem:GeneralCases}]\label{theoremB}
    Let $U\subset \widetilde{X}\cup \partial_\infty \Gamma$ be a neighborhood of $\partial_\infty {\Gamma}$.
    The sequence of the rescaled metric spaces $(\widetilde{X}\setminus U,d_t)$ uniformly converges to $(\pr_\varphi(\widetilde{X} \setminus U), d_{T(\varphi)})$ in the sense of Gromov-Hausdorff as $t$ tends to $\infty$.
\end{theorema}

By introducing the complement of a neighborhood of $\partial_\infty \Gamma$, we can characterize all domains where $d_t$ converge uniformly.
Note that the complement $\widetilde{X}\setminus U$ can be a non-compact subset.

\Cref{theoremB} implies that, in the sense of \cite[Definition 5.1]{wolf1995rtree}, the identity maps $\id\colon \widetilde{X} \to (\widetilde{X},d_t)$ converge to $\pr_\varphi \colon \widetilde{X}\to T(\varphi)$.
The identity map $\id\colon \widetilde{X} \to (\widetilde{X},d_t)$ is a harmonic map to the metric space $(\widetilde{X},{\widetilde{h}_{t\varphi}}\!^\ast\widetilde{C}_{t\varphi}/\sqrt{4t})= (\widetilde{X},d_t)$, and the limit $\pr_\varphi\colon\widetilde{X}\to T(\varphi)$ is also a harmonic map as in \cite[Proposition 3.1]{wolf1995rtree}.

Lastly, we describe an application of \Cref{theoremB}.
Let $\mathcal{C}$ and $\mathcal{A}$ be the sets of isotopy classes of simple closed curves and properly embedded arcs on $S$, respectively.
We consider the hyperbolic structures in terms of the length function $\ell_\ast \colon \Teich (S) \to \R_{\geq 0}^{\mathcal{C}\cup\mathcal{A}}$, and describe the limit along a ray.

\begin{corollarya}[\Cref{corollary:PointInBoundary}]
    $(4t)^{-1/2}\ell_\ast [C_{t\varphi},f_{t\varphi}]\to I_\ast\mathcal{F}_{\varphi,\mathrm{vert}}\in \R^{\mathcal{C}\cup\mathcal{A}}$ as $t\to \infty$.
\end{corollarya}

From this corollary, we find that the direction of the ray determines the limit point in the projective space $P(\R_{\geq 0}^{\mathcal{C}\cup\mathcal{A}})$ of functions.
Since the limit is an accumulation point of the sequence of hyperbolic surfaces $\{[C_{t\varphi},f_{t\varphi}]\}_{t>0}$, there is a point in the \textit{Thurston boundary} corresponding to the limit.
In \Cref{subsection:TheThurstonBoundary}, we define the Thurston boundary of the Teichmüller space of a surface with closed boundaries or crown ends in a similar way to \cite{alessandrini2016horofunction}.

We briefly remark on some relevant results.

Using the notion of ultralimit, we obtain the limit metric space $(X_\infty,d_\infty)$ of the sequence $(\widetilde{X},d_t)$ as described in \cite{kleiner1997rigidity}.
Furthermore, it follows from \cite{korevaar1997global} that harmonic maps $h_t\colon \widetilde{X} \to \widetilde{C}_{t\varphi}$ converge uniformly on compact subsets of $\widetilde{X}$ to a harmonic map $h_\infty \colon \widetilde{X} \to (X_\infty,d_\infty)$.
These results are obtained under a general framework.
In contrast, the result in \Cref{theoremB} is more desirable, as it provides explicit descriptions of the limits of both metric spaces and harmonic maps, with convergence valid even on the noncompact domains $\widetilde{X}\setminus U$.
Additionally, our conclusion is achieved without relying on significant tools in \cite{kleiner1997rigidity,korevaar1997global}.

Recently, for the results on which this paper is based, analogous results have been obtained for the Hitchin component of surface group representations into $\mathrm{SL}(3,\R)$.
For instance, the paper \cite{loftin2022limits} demonstrates that harmonic maps arising from rays of cubic differentials converge to a harmonic map into a building and the associated holonomy representations into $\mathrm{SL}(3,\R)$ also converge to a representation into the isometry group of the building.
Moreover, the paper \cite{nie2023poles} establishes an analogue of Gupta's work \cite{gupta2021harmonic} for convex $\R P^2$-structures.

\subsection*{Acknowledgements}
I would like to thank my supervisor, Shinpei Baba, for helpful advice, many discussions, and encouragement.
I am grateful to the anonymous referee for many helpful comments and suggestion of references.
This work was supported by JSPS KAKENHI Grant Number JP23KJ1468.

\section{Preliminaries}

\subsection{Harmonic maps}

Let $X$ be a Riemann surface with a local holomorphic coordinate $z$ and $(S,\rho(w)|dw|^2)$ be a surface with a Riemannian metric $\rho$, where $w$ denotes a isothermal coordinate of $\rho$.
A $C^2$ map $f\colon X\to (S,\rho)$ is \textit{harmonic} if $f$ satisfies
\begin{equation*}
    f_{z\overline{z}}+\frac{\rho_w}{\rho} f_zf_{\overline{z}}=0.
\end{equation*}
For a map $f\colon X\to (S,\rho)$, the $(2,0)$-part of the pullback metric $(f^\ast \rho)^{2,0}$ is a quadratic differential on $X$, which is called the \textit{Hopf differential} of $f$.
In the Teichmüller theory of harmonic maps, the following is fundamental: if the Jacobian of $f$ is nowhere vanishing, $f$ is a harmonic map if and only if its Hopf differential is a holomorphic quadratic differential on $X$ (see \cite{sampson1978harmonic}).

\subsection{Teichmüller space}\label{section:TeichmullerSpace}

The \textit{Teichmüller space} is the space of hyperbolic structures with marking on a topological surface.
In this subsection, we will define a \textit{marking} as a map from a punctured Riemann surface to a hyperbolic surface with closed geodesic boundaries or crown ends.

Let $\overline{S}$ be an oriented, closed surface of genus $g$.
From $\overline{S}$, we remove $n$ open disks $D_1,\ldots,D_n\subset \overline{S}$ whose closures are pairwise disjoint.
In addition, for each $1 \leq i\leq n$, we remove $m_i\geq 0$ distinct points on $\partial D_i$.
We denote the resulting surface by $S$.
Except for 6 cases (see \Cref{figure:ExceptionalSurfaces}), $S$ admits a complete hyperbolic structure of finite area with totally geodesic boundary.
\begin{figure}[h]
   \centering
   \begin{overpic}[scale=1.1]{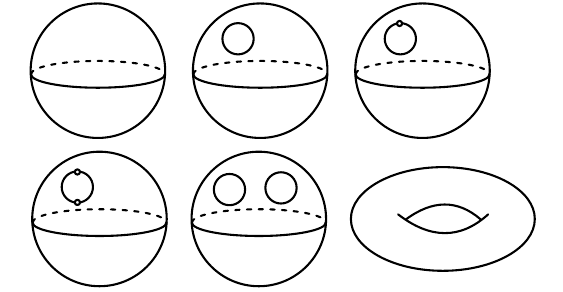}
   \end{overpic}
   \caption{Exceptional surfaces}
   \label{figure:ExceptionalSurfaces}
\end{figure}

Let $X$ be a closed Riemann surface of genus $g$ with $n$ punctures.
At each puncture $p_i\ (1\leq i\leq n)$, we pick $m_i$ tangent directions $v_{i,1},\ldots,v_{i,m_i} \in (T_{p_i}X\setminus \{0\})/\R_{\geq 0} $ counterclockwise.
The data $\theta=(\theta_i)_{i=1}^n$ of tuples $\theta_i=(v_{i,1},\ldots,v_{i,m_i})$ is called \textit{direction data}.
Taking a real oriented blow-up at each puncture $p_i$ and removing the points corresponding to the directions in the direction data, we obtain the surface $S$.
Thus, we can regard $X$ as a subset of $S$.

Let $C$ be a complete, finite area hyperbolic surface homeomorphic to $S$ with totally geodesic boundary.
A \textit{marking} of $C$ from $(X,\theta)$ is a diffeomorphism $f\colon X \to C\setminus \partial C$ which is the restriction of a diffeomorphism $\bar{f}\colon S\to S$.
For a hyperbolic surface $C$ homeomorphic to $S$ and its marking $f$ from $(X,\theta)$, we call the pair $(C,f)$ the \textit{marked hyperbolic surface}.

Let $(C_1,f_1)$ and $(C_2,f_2)$ be marked hyperbolic surfaces from $(X,\theta)$.
We define the equivalence relation $(C_1,f_1)\sim (C_2,f_2)$,
if there exists an isometry $g\colon C_1\to C_2$ such that $g\circ f_1$ is isotopic to $f_2$ relative to the direction data $\theta$. 
The \textit{Teichmüller space} $\mathcal{T}(S)$ is the quotient space of equivalence classes of marked hyperbolic surfaces homeomorphic to $S$.
The Teichmüller space $\mathcal{T}(S)$ is homeomorphic to the Euclidean space of dimension
\begin{equation*}
    6g-6+\sum_{i=1}^l(m_i+3).
\end{equation*}

\subsection{Principal part, analytic residue}

Let $X$ be a closed Riemann surface with punctures $\{p_1,\ldots,p_n\}$ and $\varphi$ a holomorphic quadratic differential on $X$ with poles at punctures.
We take a coordinate disk $(D_i,z_i)$ around each puncture $p_i$.
Let $k_i$ be the order of $\varphi$ at $p_i$.
The 1-form $\pm\sqrt{\varphi(z_i)}\,dz_i$ (up to the multiplication by $\pm 1$) can be expanded in a Laurent series.
Then the principal part of the Laurent series (i.e. terms of degree at least $-1$) is
\begin{equation} \label{eq:PrincipalPart}
    \begin{cases}
        \pm (c_r z_i^{-r}+c_{r-1}z_i^{-r+1}+\cdots+c_1z_i^{-1})\,dz_i & (k_i \text{ is even})\\
        \pm z^{-1/2} (c_r z_i^{-r}+c_{r-1}z_i^{-r+1}+\cdots+c_1z_i^{-1})\,dz_i & (k_i \text{ is odd}),
    \end{cases}
\end{equation}
where $k_i=2r$ or $2r+1$, and $c_1,\ldots,c_r$ are complex numbers.
\begin{definition}[analytic residue]
    We call the form in \Cref{eq:PrincipalPart} the \textit{principal part} of $\varphi$ at $p_i$ with respect to the coordinate disk $(D_i,z_i)$.
    The coefficient of degree $-1$ term in \Cref{eq:PrincipalPart} is determined up to the multiplication by $\pm 1$. It is called the \textit{analytic residue} $\Res(\varphi,p_i)$ of $\varphi$ at $p_i$.
\end{definition}
By this definition, if $k_i$ is odd, the analytic residue of $\varphi$ at $p_i$ is $0$, and if $k_i=2$, the analytic residue of $\varphi$ at $p_i$ equals to the square root of the coefficient of $z_i^{-2}$ in the Laurent series of $\varphi$.
By the residue theorem, we have 
\begin{equation*}
    \pm\int_\gamma \sqrt{\varphi} = \pm 2\pi i \Res (\varphi,p_i),
\end{equation*}
where $\gamma$ is a loop around $p_i$.
The principal part depends on the choice of the coordinate disk, but the analytic residue does not.

Conversely, if a form $P_i(z_i)\,dz_i$ which is expressed as \Cref{eq:PrincipalPart} is given on each coordinate disk $(D_i,z_i)$, we call the data $P=(\pm P_i(z_i))_{i=1}^n$ the \textit{prescribed principal part data}.
Let $\QD(X,P)$ denote the space of holomorphic quadratic differentials on $X$ whose principal part at each puncture $p_i$ with respect to $(D_i,z_i)$ is equal to $P_i(z_i)$. 
The analytic residue $\Res(P,p_i)$ of $P$ at $p_i$ is the coefficient of $z_i^{-1}$ in $P_i(z_i)$.

\subsection{Direction data obtained from principal parts}\label{subsection:DirectionDataObtainedFromPincipalParts}

We consider a prescribed principal part data $P=(P_i(z_i))_{i=1}^n$ on a closed Riemann surface $X$ of genus $g$ with $n$ punctures $p_1,\ldots,p_n$.
We suppose the quadratic differential $P_i^2=P_i(z_i)^2dz_i^2$ has a pole of order $m_i+2$ at each puncture $p_i$.
In the case of $m_i\geq 1$, if each leaf of horizontal foliation of $P_i^2$ on $(D_i,z_i)$ tends to $p_i$, its tangent direction converges to one of the $m_i$ asymptotic directions $v_{i,1},\ldots,v_{i,m_i}$ determined by $P_i^2$ (e.g.\ see \cite{gupta2019meromorphic}).
We set $\theta_i=(v_{i,1},\ldots,v_{i,m_i})$ and $\theta(P)=(\theta_i)_{i=1}^n$.

In the following, we suppose that the marking of every $[C,f] \in \mathcal{T}(S)$ is from $(X,\theta(P))$ to $C$. 

\subsection{Length data}
Following \cite{allegretti2021stability}, we here define the \textit{length data} of the marked hyperbolic surface $C \in \mathcal{T}(S)$.
Let $\partial_i$ denote the closed boundary or crown end of $C$ corresponding to $p_i$ by the marking of $C$.
We first consider the case that the number of cusps of $\partial_i$ is a positive even number.
On the hyperbolic surface $C$, we put horocycles centered at each ideal vertex $p_{i,j}\ (1\leq j\leq m_i)$ of $\partial_i$, and truncate ideal vertices along the horocycles.
If the horocycles are disjoint each other, a geodesic segment is left on each geodesic boundary between adjacent two ideal vertices $p_{i,j}, p_{i,j+1}$. Let $\ell_j$ denote the length of the geodesic segment.
Then, the alternate sum
\begin{equation*}
    L_i\coloneqq \pm \sum_{j} (-1)^j \ell_j
\end{equation*}
is well-defined up to the multiplication by $\pm 1$ (see \cite[Lemma 2.10]{gupta2021harmonic}).
If two adjacent horocycles have an intersection, we define the length $\ell_j$ is negative.
\begin{definition}
    The number $L_i$ up to $\pm 1$ is called the \textit{metric residue} of $C$ with respect to $\partial_i$.
\end{definition}
In the case that the number of cusps of $\partial_i$ is odd, we define $L_i=0$.
For closed boundary $\partial_i$, we define $L_i$ as the length of $\partial_i$ in the metric of $C$.

\subsection{Compatibility} 

We introduce the compatibility between the principal part data $P=(P_i)_{i=1}^n$ and the length data $L=(L_i)_{i=1}^n$.
\begin{definition}\label{definition:CompatibilityCondition}
    The length data $(L_i)_{i=1}^n$ of $C\in\mathcal{T}(S)$ is \textit{compatible} with $P$ if $|L_i|=4\pi|\Im \Res(P,p_i)|$ holds for each $i=1,\ldots,n$.
\end{definition}
\begin{remark}
    The compatibility condition is introduced by Gupta \cite{gupta2021harmonic} for hyperbolic surfaces with crown ends and by Sagman \cite{sagman2023infinite} for those with closed boundaries, and Allegretti \cite{allegretti2021stability} summarized up the both cases.
    The condition in \Cref{definition:CompatibilityCondition} is apparently different from one in \cite[p.3855 (1),(2)]{allegretti2021stability}, however we find these conditions are the same by a direct computation. (Note that the definitions of ``analytic residue'' and ``metric residue'' in papers \cite{gupta2021harmonic,sagman2023infinite,allegretti2021stability} differ slightly from each other.)
\end{remark}

\subsection{Parametrization}\label{subsection:Parametrization}

The Teichmüller space $\mathcal{T}_L(S)$ can be parametrized by the space of holomorphic quadratic differentials on $X$ via harmonic maps. In the case that the differentials have a double pole at each puncture, Wolf proved the existence of harmonic maps \cite{wolf1991infinite} and Sagman gave the parametrization result \cite{sagman2023infinite}.
In the case that the differentials have a higher order pole at each puncture, the existence and parametrization results were given by Gupta \cite{gupta2021harmonic}.
Allegretti \cite{allegretti2021stability} also considered the mixed cases, which had not been treated in the former papers.

\begin{theorem}[Existence, \cite{wolf1991infinite,gupta2021harmonic,allegretti2021stability}]
    For each marked hyperbolic surface $[C,f] \in \mathcal{T}_L(S)$, there exists a unique harmonic map $h_{[C,f]}=h\colon X\to C\setminus \partial C$ such that the Hopf differential of $h$ is in $\QD(X,P)$ and $h$ is homotopic to $f$ relative to the direction data $\theta(P)$.
\end{theorem}

\begin{theorem}[Parametrization, \cite{gupta2021harmonic,sagman2023infinite,allegretti2021stability}]\label{theorem:ParametrizationResult}
    We define the map 
    \begin{equation*}
        \Psi\colon \mathcal{T}_L(S) \to \QD(X,P)
    \end{equation*}
    by assigning each $[C,f]\in \mathcal{T}_L(S)$ to the Hopf differential of the harmonic map $h_{[C,f]}$.
    Then, the map $\Psi$ is a homeomorphism.
\end{theorem}

Let $\QD(X)$ be the space of holomorphic quadratic differentials on $X$ with poles at punctures.
By \Cref{theorem:ParametrizationResult}, given $\varphi \in \QD(X)$, we have the hyperbolic surface $C_\varphi$ and the harmonic diffeomorphism $h_\varphi\colon X\to C_\varphi$ such that $\Psi([C_\varphi,h_\varphi])=\varphi$.
Therefore, we obtain the well-defined map 
\begin{equation*}
    \rho\colon \QD(X)\to \operatorname{Met}(X) \text{ given by }\rho(\varphi)={h_\varphi}^\ast C_\varphi,
\end{equation*}
where $\operatorname{Met}(X)$ denotes the set of Riemannian metrics on $X$.
Note that $\rho(\varphi)$ can be incomplete, if a puncture is a regular point, zero or simple pole of $\varphi$.

Let $\pi\colon \widetilde{X}\to X$ be the universal covering of $X$.
The universal covering $\widetilde{X}$ is identified with the complex plane $\C$ in the case that $X$ is the Riemann sphere and the number of punctures is $1$ or $2$, and with the upper half plane $\mathbb{H}$ in the other case.
Ler $\Gamma$ be the discrete free subgroup of $\operatorname{Isom}^+(\widetilde{X})$ such that $X=\widetilde{X}/\Gamma$.
The lift metric of ${h_\varphi}^\ast C_\varphi$ to $\widetilde{X}$ is denoted by $g_\varphi$.
 
\subsection{Measured foliation and dual R-tree}\label{subsection:MeasuredFoliationAndDualR-tree}

Let $\varphi$ be a holomorphic quadratic differential on a Riemann surface $X$.
We introduce the measured foliations and the $\R$-tree obtained from $\varphi$ (for more details, see \cite{wolf1995rtree}).
In a \textit{natural coordinates} $\zeta=\xi +i\eta$ which is defined on a sufficiently small neighborhood of a nonzero point of $\varphi$,
the holomorphic quadratic differential $\varphi$ is represented by $d\zeta^2$.
A natural coordinate is uniquely determined up to a multiplication $\pm 1$ and additive constant.
Therefore, vertical lines $\{\xi=\text{constant}\}$ (resp.\,horizontal lines $\{\eta=\text{constant}\}$) determine a singular foliation on $\widetilde{X}$ which is pronged at each zero of $\varphi$.
The number of prongs at a zero is equal to $(\text{the order of the zero})+2$.
The density $|d\xi|$ (resp.\,$|d\eta|$) defines a transverse measure on the singular foliation.
The measured foliation obtained from $\varphi$ is called the \textit{vertical measured foliation} (resp.\,horizontal measured foliation) of $\varphi$ and denoted by $\mathcal{F}_{\varphi,\mathrm{vert}}$ (resp.\,$\mathcal{F}_{\varphi,\mathrm{hori}}$).

Taking the lift $\widetilde{\mathcal{F}}_{\varphi,\mathrm{vert}}$ of the vertical measured foliation of $\varphi$ to the universal covering $\widetilde{X}$, we obtain the \textit{leaf space} $T(\varphi)$ which is given by collapsing each leaf of $\widetilde{\mathcal{F}}_{\varphi,\mathrm{vert}}$ to a point.
The \textit{dual $\R$-tree} to $\widetilde{\mathcal{F}}_{\varphi,\mathrm{vert}}$ is the metric space $(T(\varphi), d_{T(\varphi)})$, where $d_{T(\varphi)}$ is the distance function on $T(\varphi)$ induced by the transverse measure $|d\xi|$.
The projection $\pr_\varphi \colon \widetilde{X} \to T(\varphi)$ is called the \textit{collapsing map}.

\subsection{Geometry of harmonic maps}
\label{subsection:GeometryOfHarmonicMaps}

Let $X$ be a Riemann surface and $Y$ be a hyperbolic surface with finite area.
We suppose that $h \colon X\to Y$ is a harmonic diffeomorphism and $\varphi$ is the Hopf differential of $h$.
The norm $|\nu(h)|\coloneqq |h_{\overline{z}}/h_z|$ of the Beltrami differential of $h$ is a well-defined function on $X$.
We set $G(h)=\log(1/|\nu(h)|)$.
The function $G(h)$ plays an important role in studying the geometry of harmonic maps.
In \cite{minsky1992harmonic}, Minsky gave useful estimates for the function $G(h)$, which originated from Wolf's work \cite{wolf1991high}.
\begin{proposition}[{\cite[Lemma 3.2]{minsky1992harmonic}}]\label{Proposition:RoughBound}
    We fix a point $p\in X$ which is nonzero of $\varphi$.
    We suppose a $|\varphi|$-metric ball $B_{|\varphi|}(p,r)$ contains no zero of $\varphi$ and is embedded into $X$.
    Then, we have
    \begin{equation*}
        G(h)(p)\leq \sinh^{-1}\left(\frac{\operatorname{Area}(X)}{2\pi r^2}\right).
    \end{equation*}
\end{proposition}

\begin{proposition}[{\cite[Lemma 3.3]{minsky1992harmonic}}]\label{Proposition:ExponentialEstimate}
    Let $p\in X$ be a point $|\varphi|$-distant at least $d$ away from zeros of $\varphi$.
    If $G(h)$ is bounded by $B$ from above on $B_{|\varphi|}(p,d)$,  then we have
    \begin{equation*}
        G(h)(p) \leq \frac{B}{\cosh d}.
    \end{equation*}
\end{proposition}

Using these propositions, we prepare some estimates we will use later.
Let $p_i$ be a puncture of $X$.
If $p_i$ is a double pole of $\varphi$, we can take a disk neighborhood $U_i$ of $p_i$ such that $(U_i\setminus \{p_i\},|\varphi|)$ is isometric to a Euclidean one-sided infinite annulus.
If $p_i$ is a pole of order higher than $2$, we can take a \textit{sink neighborhood} of $p_i$, which does not contain zeros of $\varphi$ (see Definition 2 \cite{gupta2019meromorphic}).
The sink neighborhood $U_i$ is broken into \textit{half-planes} and \textit{horizontal strips} by the horizontal measured foliation of $\varphi$.
Each half-plane is isometric to a Euclidean half plane $\{(x,y)\mid y> 0\}$ and each horizontal strip is isometric to a Euclidean strip $\{(x,y)\mid a<y<b \}$.

\begin{proposition}\label{proposition:estimateFromMinsky'sOne}
    We suppose that the length of any nontrivial loop around $p_i$ in $U_i\setminus\{p_i\}$ is at least $2C$ and $\partial U_i$ is of distance at least $C$ away from zeros of $\varphi$ in the $|\varphi|$-metric.
    Then, for each $p\in U_i\setminus\{p_i\}$,
    \begin{equation*}
        G(h)(p)\leq 2e^{-f(p)}\sinh^{-1}(\operatorname{Area}(Y)/2\pi C^2),
    \end{equation*}
    where $f(p)=\operatorname{dist}_{|\varphi|}(p,\varphi^{-1}(0))$.
\end{proposition}

\section{Uniform convergence of distance functions: ideal polygon case}

\subsection{Main result}\label{subsection:StatementIdealPolygonsCase}

First, we state and prove the main theorem in the case of ideal polygons, since this is the simplest case and many of the methods of the proof are the same as in the general case.

Let $S$ be a topological ideal $k$-gon ($k\geq 3$), i.e.\,$S$ is a topological closed disk and we remove distinct $k$ points on the boundary.
In this case, the Riemann surface $X$ is equivalent to $\mathbb{C}P^1\setminus \{\infty\} \cong \C$ and $\varphi$ is a polynomial quadratic differential on $\C$.
The point at infinity is a pole of order $k+2$, and
the degree of the polynomial differential $\varphi$ equals $k-2$.
The principal part of $\varphi$ determines the direction data at $\infty$.

For $t>0$, let $d_t$ denote the distance function associated with the metric $(4t)^{-1}g_{t\varphi}$.
In the ideal polygon case, the metric $g_{t\varphi}$ is equal to the pullback metric ${h_{t\varphi}}^\ast C_{t\varphi}$ (see the last paragraph in \Cref{subsection:Parametrization}).
Let $I_\varphi$ denote the composite function of the collapsing map and the distance function of the dual $\R$-tree $T(\varphi)$ to $\mathcal{F}_{\varphi,\mathrm{vert}}$ i.e.\ we set $I_\varphi(x_1,x_2)=d_{T(\varphi)}(\pr_\varphi x_1,\pr_\varphi x_2)$.
Note that, in the ideal polygon case, the dual $\R$-tree is a simplicial metric tree.

\begin{theorem}\label{theorem:IdealPolygonsCase}
    The family of distance functions $d_t$ uniformly converges to $I_\varphi$ on $X=\C$.
\end{theorem}

This theorem tells us that, for each $\varepsilon>0$, the collapsing map $\pr_\varphi$ is an \textit{$\varepsilon$-approximation} between $(X,d_t)$ and $(T(\varphi),d_{T(\varphi)})$ by taking a sufficiently large $t$, i.e. there exists $t_1>0$ such that, for each $t>t_1$ and each two points $x_1,x_2\in X$, 
\begin{equation}\label{eq:Approximation}
    |d_t(x_1,x_2)-d_{T(\varphi)}(\pr_\varphi x_1,\pr_\varphi x_2)|<\varepsilon.
\end{equation}

Following \cite{bridson1999metric}, we recall the definition of the Gromov-Hausdorff convergence.

\begin{definition}
    Let $(X,d)$ and $(Y,d')$ be metric spaces.
    For $\varepsilon>0$, the relation $R\subset X\times Y$ is called $\varepsilon$-\textit{relation} between $(X,d)$ and $(Y,d')$ if $R$ satisfies the following conditions:
    \begin{enumerate}
        \item the relation $R$ is surjective i.e. $\pr_1(R)=X, \pr_2(R)=Y$,
        \item if $(x_1,y_1),(x_2,y_2)\in R$, then $|d(x_1,x_2)-d(y_1,y_2)|<\varepsilon$.
    \end{enumerate}

    A sequence of metric spaces $\{(X_t,d_t)\}_{t\in \R}$ converges to a metric  space $(Y,d)$ \textit{in the sense of Gromov-Hausdorff} as $t\to \infty$, if it satisfies the following: for each $\varepsilon>0$, there exists $t_1\in \R$ such that, for every $t>t_1$, there exists an $\varepsilon$-relation $R_t$ between $(X_t,d_t)$ and $(Y,d')$.
\end{definition}

Setting $R=\{(x,y)\in X\times T(\varphi)\mid y=\pr_\varphi x\}$, we find that, for each $\varepsilon>0$, $R$ is an $\varepsilon$-relation between $(X,d_t)$ and $(T(\varphi),d_{T(\varphi)})$ for every $t>t_1$ from \Cref{eq:Approximation}. Therefore we obtain the following corollary.

\begin{corollary}\label{corollary:GromovHausdorff}
    The metric space $(X,d_t)$ converges to $(T(\varphi),d_{T(\varphi)})$ in the sense of Gromov-Hausdorff.
\end{corollary}

\subsection{One side of the inequality}
In this section, we prepare some lemmas to prove \Cref{theorem:IdealPolygonsCase}. 
For simplicity,
we denote the harmonic map $h_{t\varphi}\colon X\to C_{t\varphi}$ by $h_t$, the function $G(h_t)$ by $G(t)$ (see \Cref{subsection:GeometryOfHarmonicMaps}), and the metric $(4t)^{-1}g_{t\varphi}$ by $g_t$.

\begin{lemma}\label{lemma:IntersectionNumbersBoundedFromDistance}
    For each $t>0$, the inequality $I_\varphi<d_t$ holds on $X$.
\end{lemma}

\begin{proof}
    We take any two points $x_1,x_2\in X$.
    By the definition of the distance on the dual $\R$-tree, we have
    \begin{equation*}
        I_\varphi(x_1,x_2)=\inf\left\{\int_{\gamma} |d\xi| \,\middle|\, \gamma \text{ is a curve connecting } x_1 \text{ and }x_2 \right\},
    \end{equation*}
    where $\zeta=\xi+i\eta$ is natural coordinates of $\varphi$.
    In the coordinates $\zeta$, the metric $g_t$ is represented by 
    \begin{equation}
        \label{eq:LocalRepresentationOftheMetric}
        ds_{g_t}^2 = \frac{\cosh G(t) +1}{2}\,d\xi^2+\frac{\cosh G(t) -1}{2}\,d\eta^2
    \end{equation}
    (see \cite[Section 3]{minsky1992harmonic}).
    From this representation, since we find $ds_{g_t}^2\geq d\xi^2$.
    Let $\gamma_t\colon [0,1]\to X$ be a $g_t$-geodesic segment from $x_1$ to $x_2$.
    Since $\int|d\xi|$ defines the transverse measure of the vertical measured foliation of $\varphi$, for each $t>0$, we have
    \begin{equation*}
        d_t(x_1,x_2)=\ell_{g_t}(\gamma_t)=\int_\gamma ds_{g_t}\geq \int_\gamma|d\xi|\geq I_\varphi(x_1,x_2).
    \end{equation*}
    Thus, we obtain the conclusion.
\end{proof}

\begin{remark}
    Since the above proof works similarly, we find the following: Let $(S,\rho)$ be a surface with a convex metric, $h\colon X\to (S,\rho)$ be a harmonic diffeomorphism and $\varphi$ be its Hopf differential.
    Then, the inequality $d_{h^\ast \rho}> I_\varphi$ holds on $X$, where $d_{h^\ast \rho}$ denotes the distance function associated with the pullback metric $h^\ast \rho$.
\end{remark}

\subsection{Estimate for vertical and horizontal lines}

We next estimate the hyperbolic length of a vertical or horizontal line of $\varphi$.

\begin{lemma}\label{lemma:estimateForInfintiteLines}
    Let $U$ be a sufficiently small sink neighborhood (see \Cref{subsection:GeometryOfHarmonicMaps}) of $\infty$ with respect to $\varphi$ so that the length of any nontrivial loop around $\infty$ in $U\setminus \{\infty\}$ is at least $2C$ and $\partial U$ is of distance at least $C$ away from $\varphi^{-1}(0)$ in the $|\varphi|$-metric.
    Then, there exists a sufficiently large $t_0>0$ such that, for every $t>t_0$, there exists a constant $D>0$ depending only on $\varphi$ such that the following hold:
    \begin{enumerate}
        \item Let $H\colon [0,\infty)\to X$ be a horizontal line of $\varphi$, i.e.\ a curve whose image is in a horizontal leaf of $\varphi$.
        We suppose that $H$ is parametrized by the $|\varphi|$-arc length and is contained in $U$.
        Then, for each $a\geq0$, 
        \begin{equation}
            \ell_{g_t}(H|_{[0,a]})-i(\mathcal{F}_{\varphi,\mathrm{vert}},H|_{[0,a]}) <De^{-C\sqrt{t}}\label{eq:HorizontalInfiniteSegment}
        \end{equation}
        \item Let $V\colon [0,\infty)\to X$ be a vertical line of $\varphi$, i.e.\ a curve whose image is in a vertical leaf of $\varphi$.
        We suppose that $V$ is parametrized by the $|\varphi|$-arc length and contained in $U$. Then,
        \begin{equation}
            \ell_{g_t}(V) < De^{-C\sqrt{t}}. \label{eq:VerticalInfiniteSegment}
        \end{equation}
    \end{enumerate}
\end{lemma}

\begin{proof}
    We may assume $\Gamma(s)$ is a maximal horizontal (or vertical) line that is contained in $U$.
    If $\Gamma(s)$ is contained in a half-plane domain, it is a bi-infinite line, and if $\Gamma(s)$ is contained in a strip domain, it is a one-sided infinite line.

    We first consider the case that $\Gamma(s)$ is a horizontal line.
    Let $P_1,\ldots,P_{k}$ be maximal half-plane domains bounded by vertical lines contained in $U$ (in \Cref{figure:Half-PlaneDomain}, these are domains bounded by dotted lines).
    We cyclically put the index for $P_1,\ldots,P_k$.
    Then, there exist adjacent $P_i$ and $P_{i+1}$ decomposing $\Gamma(s)$ into a finite arc and at most two one-sided infinite lines.
    The $|\varphi|$-length of the finite arc is at most $2W$ of the width of strip domains which lies between $P_i$ and $P_{i+1}$.
    Let $\Gamma_1\colon [0,\infty)\to U$ be the one-sided infinite horizontal line.
    Then, the function $f\circ\Gamma_1(s)=d_{|\varphi|}(\Gamma_1(s),\varphi^{-1}(0)) $ is larger than $C+s$. 
    In the case that $\Gamma(s)$ is a vertical line as well, the same thing holds true.
    
    \begin{figure}
        \centering
        \begin{overpic}[scale=0.8]{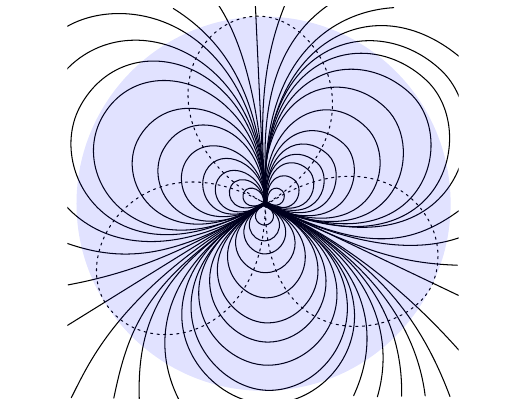}
        \end{overpic}
        \caption{The blue colored domain is $U$. The dotted lines denote the vertical lines bounding maximal half-plane domains.}
        \label{figure:Half-PlaneDomain}
    \end{figure}

    Here we introduce an inequality used later: for any constant $B>0$, there exist constants $A>0$ and $x_0>0$ such that
    \begin{equation}\label{eq:BasicInequality}
        \sqrt{\cosh(B e^{-x})-1} \leq A e^{-x}\ \ \ \ (x>x_0).
    \end{equation}
    Since this inequality is equivalent to
    \begin{equation*}
        B\leq e^x\cosh^{-1}(A^2e^{-2x}+1),
    \end{equation*}
    and the right hand side converges to $\sqrt{2}A$ as $x\to \infty$, the inequality \Cref{eq:BasicInequality} holds by taking $A=B/\sqrt{2}+1$.

    We here consider the case where $\Gamma$ is an infinite vertical line $V$, and for simplicity we suppose that $\Gamma$ is a bi-infinite line.
    Then, setting $B=2\sinh^{-1}((k-2)\pi/2\pi C^2)$, we have 
    \begin{align*}
        \ell_{g_t}(V) &= \frac1{\sqrt{2}}\int_V \sqrt{\cosh G(t)-1}\,d\eta \\
        &\leq \frac1{\sqrt{2}}\int_V \sqrt{\cosh(Be^{-f \sqrt{t}} )-1}\,d\eta 
        & (\text{\Cref{proposition:estimateFromMinsky'sOne}}) \\
        &\leq \frac{A}{\sqrt{2}} \int_V e^{-f\sqrt{t}}\,d\eta & \text{\Cref{eq:BasicInequality}},
    \end{align*}
    where the final inequality holds if $t>(x_0/C)^2$, since $f\circ V(s)>C$.
    We now have 
    \begin{align*}
        \int_V e^{-\sqrt{t}f} \,d\eta &=\int_{-W}^W e^{-f\circ V(s)\sqrt{t}}\,ds+2\int_W^\infty e^{-f\circ V(s)\sqrt{t}}\,ds \\
        &\leq \int_{-W}^W e^{-C\sqrt{t}}\,ds+2\int_W^\infty e^{-(s+C)\sqrt{t}}\,ds\\
        &\leq e^{-C \sqrt{t}} \left(2W+\frac{1}{\sqrt{t}}\right).
    \end{align*}
    We thus obtain \Cref{eq:VerticalInfiniteSegment}.

    In the case where $\Gamma$ is a horizontal line $H$, we find
    \begin{align*}
        \ell_{g_t}(H|_{[0,a]}) &= \frac{1}{\sqrt{2}} \int_{H|_{[0,a]}} \sqrt{\cosh G(t) +1}\,d\xi \\
        &\leq \frac1{\sqrt{2}}\int_{H|_{[0,a]}} \sqrt{\cosh G(t)-1}\,d\xi+\int_{H|_{[0,a]}}\,d\xi,
    \end{align*}
    since $\sqrt{\cosh G(t) +1}<\sqrt{\cosh G(t)-1}+\sqrt{2}$.
    Therefore we have
    \begin{equation*}
        \ell_{g_t}(H|_{[0,a]})-i(\mathcal{F}_{\varphi,\mathrm{vert}},H|_{[0,a]})\leq \frac1{\sqrt{2}}\int_{H|_{[0,a]}} \sqrt{\cosh G(t)-1}\,d\xi.
    \end{equation*}
    Since we can evaluate the right hand side in a similar manner as the case of a vertical line, \Cref{eq:HorizontalInfiniteSegment} holds.
\end{proof}

We also have the following estimate, which is stated in \cite{wolf1991high} and, it can be proven in a similar way as the above proof.

\begin{lemma}[{\cite[Lemma 2.3]{wolf1991high}}]
    \label{lemma:estimateForCompactArcs}
    Let $H$ and $V$ be a compact horizontal and vertical arc of $\varphi$, respectively.
    If $H$ and $V$ avoid an $\varepsilon$-neighborhood of $\varphi^{-1}(0)$ in the $|\varphi|$-metric, then 
    \begin{equation*}
        \ell_{g_t}(H)\leq (1+De^{-\varepsilon \sqrt{t}})\ell_{|\varphi|}(H),\ \  \ell_{g_t}(V)\leq De^{-\varepsilon \sqrt{t}}\ell_{|\varphi|}(V) ,
    \end{equation*}
    where the constant $D$ depends only on $\varepsilon$.
\end{lemma}

\begin{lemma}\label{lemma:DiameterOfEpsilonNeighborhoodOfzero}
    Let $p\in X$ be a zero of $\varphi$.
    We take $\varepsilon>0$ so that no zero of $\varphi$ is contained in the $|\varphi|$-ball $B_{|\varphi|}(p,3\varepsilon)$ of radius $3\varepsilon$ centered at $p$.
    Then, for a sufficiently large $t$, the diameter of $B_{|\varphi|}(p,\varepsilon)$ in the $g_t$-metric is at most $M\varepsilon$, where $M$ is a constant depending only on the degree of the zero $p$.
\end{lemma}

\begin{proof}
    Let $P$ be the $|\varphi|$-regular right angled polygonal region tangent to $B_{|\varphi|}(p,\varepsilon)$ (see \Cref{figure:RightAngledPolygon}).
    \begin{figure}
       \centering
       \begin{overpic}[scale=1.2]{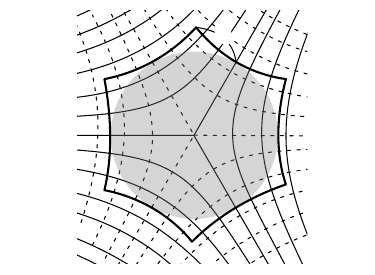}
           \put(58.5,62.5){$\varepsilon$}
       \end{overpic}
       \caption{The dotted lines denote the vertical lines of $\varphi$.
       The thick line denotes the boundary of $P$.}
       \label{figure:RightAngledPolygon}
    \end{figure}
    It is sufficient to show that the $g_t$-length of $\partial P$ is at most $M\varepsilon$.
    From the assumption, each horizontal edge $H$ or vertical edge $V$ of $\partial P$ is of distance $\varepsilon$ away from zeros of $\varphi$.
    Therefore, from \Cref{lemma:estimateForCompactArcs}, we find 
    \begin{equation*}
        \ell_{g_t}(H)\leq 2(1+De^{-\varepsilon\sqrt{t}})\varepsilon,\ \ \ \ell_{g_t}(V)\leq 2De^{-\varepsilon\sqrt{t}}\varepsilon.
    \end{equation*} 
    Therefore, the $g_t$-length of $\partial P$ is bounded by
    \begin{equation*}
        \ell_{g_t}(\partial P)\leq 2\varepsilon \ord_\varphi(p)(1+2De^{-\varepsilon\sqrt{t}}).
    \end{equation*}
    We thus get the conclusion.
\end{proof}

\subsection{Proof in the case of ideal polygons}

\begin{proof}[Proof of \Cref{theorem:IdealPolygonsCase}]
    We fix a sufficiently small $0<\varepsilon<C$ so that, for any two zeros $z_i,z_j$ of $\varphi$, the disks $B_{|\varphi|}(z_i,2\varepsilon)$ and $B_{|\varphi|}(z_j,2\varepsilon)$ have no intersection.
    We pick any two points $x_1,x_2 \in X$ and take any path $\Gamma_0$ from $x_1$ to $x_2$, and
    we deform $\Gamma_0$, in several steps, via homotopy relative to the endpoints.
    (Step 1) We straighten $\Gamma_0$ to $|\varphi|$-geodesic $\Gamma_1$.
    (Step 2) We deform $\Gamma_1$ to a $\varphi$-\textit{staircase curve} $\Gamma_2$ outside the $\varepsilon$-neighborhood of zeros of $\varphi$.
    Here, the $\varphi$-staircase curve is a piecewise horizontal or vertical arc of $\varphi$.
    This deformation keeps horizontal and vertical measures of the curve.
    (Step 3) We drag out the horizontal measure of $\Gamma_2$ in the $\varepsilon$-neighborhood of each zero which $\Gamma_2$ passes through.
    Since $\Gamma_2$ is passing through at most finitely many zeros of $\varphi$, the number $N$ of times dragging out is also finite.
    This deformation preserves the horizontal measure of $\Gamma_2$, although the vertical measure increases in general.
    However, the additional vertical measure is at most $4N\varepsilon$.
    Let $\Gamma_3$ be the resulting curve.
    (Step 4) We replace each vertical arc contained in $\varepsilon$-neighborhoods of zeros by the $g_t$-geodesic segment whose endpoints are the original endpoints connected by the vertical arc (see \Cref{fig:StepsOfTheDeformation2}).
    \begin{figure}
        \centering
        \begin{overpic}[scale=1.1]{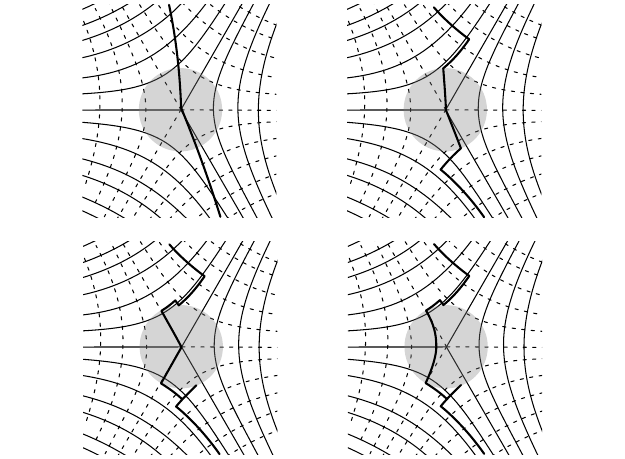}
        \end{overpic}
        \caption{This illustrates the steps of the deformation around zeros. The dotted lines are leaves of the vertical foliation, the solid lines are leaves of the horizontal foliation. The upper left side: Step 1, the upper right side: Step 2, the lower left side: Step 3, the lower right side: Step 4.}
        \label{fig:StepsOfTheDeformation2}
     \end{figure}
    \Cref{lemma:DiameterOfEpsilonNeighborhoodOfzero} implies the length of the geodesic in the $g_t$-metric is bounded by $M\varepsilon$.
    We denote the resulting curve by $\Gamma_{4,t}$.

    Now $\Gamma_{4,t}$ can be decomposed into horizontal line parts $\Gamma_{\mathrm{hori}}$, vertical line parts $\Gamma_{\mathrm{vert}}$ and $g_t$-geodesics part $\Gamma_t$.
    Since the complement $X\setminus U$ of the sink neighborhood $U$ is compact, the horizontal or vertical measure between any two points $x_3,x_4 \in X\setminus U$ can be uniformly bounded by $K$.
    For the horizontal part, we have
    \begin{align}
        \ell_{g_t}(\Gamma_{\mathrm{hori}}) &=\ell_{g_t}(\Gamma_{\mathrm{hori}}\cap (X\setminus U)) + \ell_{g_t}(\Gamma_{\mathrm{hori}}\cap U) \notag \\
        &\leq i(\mathcal{F}_{\varphi,\mathrm{vert}},\Gamma_{\mathrm{hori}}\cap (X\setminus U))(1+De^{-\varepsilon\sqrt{t}}) \notag & (\text{by \Cref{lemma:estimateForCompactArcs}})\\
        &\quad + i(\mathcal{F}_{\varphi,\mathrm{vert}},\Gamma_{\mathrm{hori}}\cap U)+De^{-C\sqrt{t}} & \text{(by \Cref{eq:HorizontalInfiniteSegment})}\notag \\
        &\leq i(\mathcal{F}_{\varphi,\mathrm{vert}},\Gamma_{\mathrm{hori}})+(K+1)De^{-\varepsilon\sqrt{t}} \notag \\
        &= i(\mathcal{F}_{\varphi,\mathrm{vert}},\Gamma_1) + (K+1)De^{-\varepsilon\sqrt{t}} \notag \\
        &= I_\varphi(x_1,x_2) + (K+1)De^{-\varepsilon\sqrt{t}} \label{eq:Horizontalestimate}. 
    \end{align}
    For the vertical part, we see
    \begin{align}
        \ell_{g_t}(\Gamma_{\mathrm{vert}})&=\ell_{g_t}(\Gamma_{\mathrm{vert}}\cap (X\setminus U)) + \ell_{g_t}(\Gamma_{\mathrm{vert}}\cap U) \notag \\
        &\leq De^{-\varepsilon\sqrt{t}} (K+4N\varepsilon) + De^{-C\sqrt{t}} \notag \\
        & (\text{by \Cref{lemma:estimateForCompactArcs} and \Cref{eq:VerticalInfiniteSegment}}) \notag \\
        &\leq (K+4N\varepsilon+1)De^{-\varepsilon\sqrt{t}}, 
        \label{eq:Verticalestimate}
    \end{align}
    where $N$ denotes the number of zeros through which the $|\varphi|$-geodesic between $x_1$ and $x_2$ passes.
    For the geodesic part, we find 
    \begin{equation}
        \ell_{g_t}(\Gamma_t)<NM\varepsilon. \label{eq:GeodesicsEstimate}
    \end{equation}
    Since the total number $N'$ of zeros of $\varphi$ bounds $N$ from above, by combining \Cref{eq:Horizontalestimate,eq:Verticalestimate,eq:GeodesicsEstimate} and \Cref{lemma:IntersectionNumbersBoundedFromDistance}, we have 
    \begin{equation*}
        \begin{split}
            I_\varphi(x_1,x_2)&<d_t(x_1,x_2) \\
            &<\ell_{g_t}(\Gamma_{4,t}) \\
            &<I_\varphi(x_1,x_2)+(2K+4N'\varepsilon+2)De^{-\varepsilon\sqrt{t}}+N'M\varepsilon.
        \end{split}
    \end{equation*}
    This implies, for every sufficiently large $t$, 
    \begin{equation*}
        |I_\varphi(x_1,x_2)-d_t(x_1,x_2)|<\varepsilon\ \ \ \ (\forall x_1,x_2 \in X).
    \end{equation*}
    Thus, we obtain the theorem.
\end{proof}

\section{Attaching the Cayley graph ends}
\label{section:HyperbolicPlaneWithTheEndsOfTheCayleyGraph}

In this section, to state the main theorem for a Riemann surface with a nontrivial fundamental group, we explain how to attach the ends of the Cayley graph to the universal covering of the Riemann surface.
Let $X$ be a closed Riemann surface of genus $g$ with $n>0$ punctures.
If $(g,n)=(0,1)$ or $(0,2)$, the Riemann surface $X$ admits a Euclidean structure; otherwise, it admits a hyperbolic structure.
In the following, we regard $X$ as the complete hyperbolic (or Euclidean) surface.
We fix the universal covering $\pi\colon \mathbb{H}\  (\text{or } \mathbb{E})\to X$.
Let $\mathcal{L}$ denote the set of bi-infinite geodesics on $\mathbb{H}$ (or $\mathbb{E}$) which are lifts of \textit{geodesic ideal arcs} on $X$, where a geodesic ideal arc is an infinite hyperbolic (or Euclidean) geodesic connecting (possibly the same) two punctures of $X$ (if $X$ is a Euclidean cylinder, we suppose that geodesic ideal arcs on $X$ are orthogonal to the core geodesic of $X$).

Let $\mathbf{e}=\{e_1,\ldots,e_{2g+n-1}\}$ be a family of geodesic ideal arcs on $X$ such that the complement of the union $\bigcup_i e_i$ is a simply connected domain.
Let $\Gamma< \operatorname{Isom}^+(\mathbb{H})$ (or $\operatorname{Isom}^+(\mathbb{E})$) be a discrete free subgroup with $X= \mathbb{H}/\Gamma$ (or $\mathbb{E}/\Gamma$).
We fix a base point $p\in X\setminus \bigcup_i e_i$.
For each edge $e_i$, there exists a unique $\gamma_i\in\pi_1(X,p)$ such that the algebraic intersection number of $\gamma_i$ and $e_j$ is $+1$ if $i=j$ or $0$ if $i\neq j$.
The loops $\gamma_1,\ldots,\gamma_{2g+n-1}$ generate the fundamental group $\pi_1(X,p)$.
Let $G_{\mathbf{e}}$ denote the Cayley graph of the group $\pi_1(X,p)$ with respect to the generating set $\{\gamma_1,\ldots,\gamma_{2g+n-1}\}$ and we equip the graph $G_{\mathbf{e}}$ with the word metric.
We define the word distance on $G_{\mathbf{e}}$.
The \textit{Gromov boundary} $\partial_\infty G_{\mathbf{e}}$ of $G_{\mathbf{e}}$ is the set of ends of the Cayley graph, i.e.
\begin{equation*}
    \partial_\infty G_{\mathbf{e}}=\{r\colon [0,\infty) \to G_{\mathbf{e}} \mid r \text{ is a geodesic ray in }G_{\mathbf{e}} \}/\sim,
\end{equation*}
where the equivalence relation $r_1\sim r_2$ is defined by that $d(r_1(t),r_2(t))$ is bounded from above for any $t\in[0,\infty)$.

The Gromov boundary $\partial_\infty G_{\mathbf{e}}$ usually has the \textit{cone topology} which is generated by the system $\{U(r,t,k)\}$ of subsets.
Here, for each ray $r\colon [0,\infty)\to G_{\mathbf{e}}$, and for $t,k>0$, we define
\begin{equation*}
    U(r,t,k)\coloneqq \{[r_1]\in\partial_\infty G_{\mathbf{e}} \mid r_1(0)=e, d(r_1(t),r(t)) < k\}.
\end{equation*}
It is known that the Gromov boundary $\partial_\infty G_{\mathbf{e}}$ is homeomorphic to a Cantor set (e.g. see \cite[p.266]{bridson1999metric}).

However we define the same topology in another way.
Let $\iota\colon G_{\mathbf{e}} \hookrightarrow \mathbb{H}$ (or $\mathbb{E}$) be a $\pi_1(X,p)$-equivariant embedding.
We define a set $\mathcal{P}$ of half-planes in $\mathbb{H}$ (or $\mathbb{E}$) by
\begin{equation*}
    \mathcal{P}=\{P\subset \mathbb{H}\ (\text{or } \mathbb{E}) \mid P=(\text{connected component of } \mathbb{H}\setminus \ell\  (\text{or } \mathbb{E}\setminus \ell)), \ell \in \mathcal{L}\}
\end{equation*}
Here we consider the set $\widehat{\mathbb{H}} = \mathbb{H}\sqcup \partial_\infty G_{\mathbf{e}}$ (or $\widehat{\mathbb{E}}\coloneqq  \mathbb{E}\sqcup \partial_\infty G_{\mathbf{e}}$), which is topologized in the following way: for $P\in \mathcal{P}$, we define the set $U_P\subset \widehat{\mathbb{H}}$ (or $\widehat{\mathbb{E}}$) as 
\begin{equation*}
    U_P = P \cup \{[r]\in\partial_\infty G_{\mathbf{e}} \mid \exists T>0 \text{ s.t. } \forall t >T, \iota\circ r(t) \in P\},
\end{equation*}
and let the union of $\{U_P\}_{P\in\mathcal{P}}$ and all open sets of $\mathbb{H}$ (or $\mathbb{E}$) be an open basis.



Next, we will show that the topology on $\widehat{\mathbb{H}}$ is independent of the choice of the family $\mathbf{e}$ of geodesic ideal arcs, the basepoint $p$, and the embedding $\iota$.
In a similar way, we can check the same fact holds for $\widehat{\mathbb{E}}$.
Let $\mathbf{e}_j=\{e_{1,j},\ldots,e_{2g+n-1,j}\}\ (j=1,2)$ be a family of geodesic ideal arcs such that the complement of the union $\bigcup_i e_{i,j}$ is a single simply connected domain.
We choose $p_j$ as the basepoint of the fundamental group of $X$.
Let $\iota_j\colon G_{\mathbf{e}_j} \hookrightarrow \mathbb{H}\ (\text{or } \mathbb{E})$ be a $\pi_1(X,p_j)$-equivariant embedding.
Let $\widehat{\mathbb{H}}_j$ denote the space $\mathbb{H} \sqcup \partial_\infty G_{\mathbf{e}_j}$ which has a topology determined from $\mathbf{e}_j, p_j,$ and $\iota_j$.
We construct a homeomorphism $\mathcal{I}\colon \widehat{\mathbb{H}}_1 \to \widehat{\mathbb{H}}_2$.
Let $c$ be a path from $p_1$ to $p_2$ and $\beta_c \colon \pi_1(X,p_1)\to \pi(X,p_2)$ be the change-of-basepoint map defined by $[\gamma]\mapsto[c\cdot \gamma\cdot c^{-1}]$.
Let $f\colon G_{\mathbf{e}_1} \to G_{\mathbf{e}_2}$ be the extension map of $\beta_c$.
Since $f$ is a quasi-isometry, it induces the homeomorphism $\partial f\colon \partial_\infty G_{\mathbf{e}_1}\to \partial_\infty G_{\mathbf{e}_2}$.
We define the map $\mathcal{I}\colon \widehat{\mathbb{H}}_1\to \widehat{\mathbb{H}}_2$ by
\begin{equation*}
    \mathcal{I}(x)=
    \begin{cases}
        x & (x\in \mathbb{H}) \\
        \partial f(x) & (x\in \partial_\infty G_{\mathbf{e}_1}).
    \end{cases}
\end{equation*}

\begin{proposition}\label{proposition:TheTopologyIsIndependentOfChoice}
    The map $\mathcal{I}\colon \widehat{\mathbb{H}}_1\to \widehat{\mathbb{H}}_2$ is a homeomorphism.
\end{proposition}

\begin{proof}
    It is sufficient to show the continuity of $\mathcal{I}$ at each point $x\in\partial_\infty G_{\mathbf{e}_1}$.
    Let $r_1\colon [0,\infty)\to G_{\mathbf{e}_1}$ be a representative of $x$.
    We need to show that if $U_P$ is a neighborhood of $[f\circ r_1]$, then it is also a neighborhood of $[r_1]$.
    This assumption means that there exists $T>0$ such that, for any $t>T$, $\iota_2\circ f\circ r_1(t) \in U_P$.
    We set $\widetilde{p}_1=\iota_1(e)\in\pi^{-1}(p_1)$.
    The embedding $\iota_1$ maps vertices of $G_{\mathbf{e}_1}$ to the orbit $\Gamma\cdot \widetilde{p}_1$.
    We define another embedding $\tau_2$ of $G_{\mathbf{e}_2}$ by $\tau_2=\iota_1\circ\beta_c^{-1}$.
    In other words, the graph $\tau_2(G_{\mathbf{e}_2})$ is obtained by changing the edges of the graph $\iota_1(G_{\mathbf{e}_1})$.
    Let $\mathcal{G}_1, \mathcal{G}_2$ denote the sets of generators determined from edge systems $\mathbf{e}_1, \mathbf{e}_2$, respectively.
    Then, the lipschitz constants of $\iota_1,\tau_2$ are given by
    \begin{align*}
        \operatorname{Lip}(\iota_1) &=\max_{\gamma\in\mathcal{G}_1} d_{\mathbb{H}} (\widetilde{p}_1,\gamma\widetilde{p}_1) \\
        \operatorname{Lip}(\tau_2) &=\max_{\gamma\in\mathcal{G}_2} d_{\mathbb{H}} (\widetilde{p}_1,\gamma\widetilde{p}_1).
    \end{align*}
    Since the images $\tau_2(G_{\mathbf{e}_2}), \iota_2(G_{\mathbf{e}_2})$ of Cayley graphs embedded into $\mathbb{H}$ are mapped to a $(2g+n-1)$-bouquet embedded into $X$ by $\pi$, there exists a constant $C>0$ such that $d_{\mathbb{H}}(\tau_2(x),\iota_2(x))<C$ for any point $x\in G_{\mathbf{e}_2}$.
    Since $f\colon G_{\mathbf{e}_1}\to G_{\mathbf{e}_2}$ is an extension of $\beta_c$, for each vertex $\gamma\in \pi_1(X,p_1) \subset G_{\mathbf{e}_1}$, we see $f(\gamma)=\beta_c(\gamma)\in \pi_1(X,p_2)$.
    Here, we prove that there exists a $K>0$ such that 
    \begin{equation}
        d_{\mathbb{H}}(\iota_2\circ f\circ r_1(t),\iota_1\circ r_1(t))< K\ \ \ \ (\forall t\geq 0).
    \end{equation}
    First, for $t=i\in\Z_{\geq 0}$, the point $\gamma_i\coloneqq r_1(i)$ is a vertex of the Cayley graph $G_{\mathbf{e}_1}$.
    Therefore, we have
    \begin{align*}
        d_{\mathbb{H}}(\iota_2\circ f\circ r_1(i),\iota_1\circ r_1(i)) &= d_{\mathbb{H}}(\iota_2\circ \beta_c(\gamma_i), \tau_2\circ \beta_c(\gamma_i))<C.
    \end{align*}
    In the case of $t\in(i,i+1)\ (i\in\Z_{\geq 0})$, we have 
    \begin{align*}
         &d_{\mathbb{H}}(\iota_2\circ f\circ r_1(t),\iota_1\circ r_1(t)) \\
        <&d_{\mathbb{H}}(\iota_2\circ f\circ r_1(t),\tau_2\circ f\circ r_1(t))+d_{\mathbb{H}}(\tau_2\circ f\circ r_1(t), \iota_1\circ r_1(t)) \\
        <&C+d_{\mathbb{H}}(\tau_2\circ f\circ r_1(t), \iota_1\circ r_1(t)) \\
        <&C+d_{\mathbb{H}}(\tau_2\circ f\circ r_1(t), \tau_2\circ f\circ r_1(i)) + d_{\mathbb{H}}(\tau_2\circ f\circ r_1(i), \iota_1\circ r_1(t))\\
        =&C+d_{\mathbb{H}}(\tau_2\circ f\circ r_1(t), \tau_2\circ f\circ r_1(i)) + d_{\mathbb{H}}(\iota_1\circ r_1(i), \iota_1\circ r_1(t)) \\
        <&C+(\operatorname{Lip}(\tau_2)\operatorname{Lip}(f)+\operatorname{Lip}(\iota_1))|t-i| \\
        <&C+\operatorname{Lip}(\tau_2)\operatorname{Lip}(f)+\operatorname{Lip}(\iota_1)\eqqcolon K.
    \end{align*}
    This implies $\iota_1\circ r_1(t)\in U_P$ for sufficiently large $t>0$.
    Therefore, we obtain the conclusion.
\end{proof}

From \Cref{proposition:TheTopologyIsIndependentOfChoice}, it follows that the space $\widehat{\mathbb{H}}$ (or $\widehat{\mathbb{E}}$) is independent of the choice of the family $\mathbf{e}$ of ideal arcs, the basepoint $p$ and the embedding $\iota$.
Thus, we can write $\widehat{\mathbb{H}}=\mathbb{H}\cup \partial_\infty \Gamma$ (or $\widehat{\mathbb{E}}=\mathbb{E}\cup \partial_\infty\Gamma$).

\begin{proposition}
    Let $F\subset \mathbb{H}$ be a fundamental domain of $\Gamma$ which is a hyperbolic ideal polygon.
    Then, the complement of a neighborhood of $\partial_\infty \Gamma \subset \widehat{\mathbb{H}}$ can be covered by the image of $F$ by finite elements of $\Gamma$.
\end{proposition}

\begin{proof}
    Let $U\subset \widehat{\mathbb{H}}$ be a neighborhood of $\partial_\infty \Gamma$.
    It is known that $\partial_\infty \Gamma$ is homeomorphic to the Cantor set, in particular, it is compact.
    Therefore, there exist finite half-planes $P_1,\ldots,P_n$ such that $ \partial_\infty \Gamma\subset U_{P_1} \cup \cdots\cup U_{P_n} \subset U$.
    Clearly the complement of $U_{P_1} \cup \cdots\cup U_{P_n}$ can be covered by the image of $F$ by finite elements of $\Gamma$.
\end{proof}

\section{Uniform convergence of distance functions: general case}

In this section, we generalize \Cref{theorem:IdealPolygonsCase} to the case that the underlying surface $S$ is not simply connected.
We take $n$ open disks $D_1,\ldots,D_n$ in an oriented, closed surface of genus $g$ such that the closures $\overline{D_1},\ldots,\overline{D_n}$ are pairwise disjoint.
We remove $D_1,\ldots,D_n$ and distinct $m_i\geq 0$ points on the boundary of $D_i$ from the surface.
Let $S$ be the resulting surface.

Let $X$ be a closed Riemann surface of genus $g$ with $n$ punctures $p_1,\ldots,p_n$.
Let $\varphi$ be a holomorphic quadratic differential on $X$ which has a pole of order $m_i+2$ at each puncture $p_i$.
Let $g_t$ denote the rescaled metric $(4t)^{-1}g_{t\varphi}$ on $\widetilde{X}$ as in \Cref{subsection:StatementIdealPolygonsCase} and $d_t$ denote the distance function determined from $g_t$ on $\widetilde{X}$.
Let $\pr_\varphi\colon \widetilde{X} \to T(\varphi)$ be the collapsing map of the dual $\R$-tree to the vertical measured foliation of $\widetilde{\varphi}$.
We set $I_\varphi = d_{T(\varphi)}\circ \operatorname{Pr}_\varphi$, where $\operatorname{Pr}_\varphi(x_1,x_2)=(\pr_\varphi x_1,\pr_\varphi x_2)$ for $x_1,x_2\in \widetilde{X}$.

\begin{theorem}\label{theorem:GeneralCases}
    Let $U\subset \widehat{\mathbb{H}}$ (or $\widehat{\mathbb{E}}$) be a neighborhood of $\partial_\infty \Gamma \subset \widehat{\mathbb{H}}$ (or $\widehat{\mathbb{E}}$).
    Then the metric functions $d_t$ uniformly converge to $I_\varphi$ on $\widetilde{X}\setminus U$.
\end{theorem}

In order to prove this, we prepare the proposition on the fundamental domain.
Let $\Gamma$ be a discrete free subgroup of $\operatorname{Isom}^+(\widetilde{X})$ such that $\widetilde{X}/\Gamma=X$.
Let $\mathbf{e}=\{e_1,\ldots,e_{2g+n-1}\}$ be a family of geodesic ideal arcs on $X$ such that $X\setminus \bigcup_i e_i$ is a single simply connected domain.
Let $F$ be a fundamental domain bounded by lifts of $e_1,\ldots,e_{2g+n-1}$ to $\widetilde{X}$.
Taking \textit{straight} $|\varphi|$-geodesic ideal arcs homotopic (rel.\ endpoints) to the geodesic ideal arcs $e_1,\ldots,e_{2g+n-1}$, we obtain the $|\varphi|$-convex fundamental domain $F_\varphi$ of $\Gamma$.
Here, a $|\varphi|$-geodesic ideal arc is said to be straight, if it does not spiral towards any puncture at the ends of the ideal arc.
Around a double pole of $\varphi$, there may exist $|\varphi|$-geodesic ideal arcs that spiral towards the pole.
\begin{proposition}\label{proposition:FiniteCopies}
    If a domain is covered by the image of $F$ by finite elements of $\Gamma$, then it is also covered by the image of $F_\varphi$ by finite elements of $\Gamma$.
\end{proposition}

\begin{proof}[Proof of \Cref{proposition:FiniteCopies}]
    It is sufficient to show that a single fundamental domain $F$ can be covered by the image of $F_\varphi$ by finite elements of $\Gamma$. 
    We suppose that it is not true.
    Then, there exist an edge $e$ of $F$, an edge $e'$ of $F_\varphi$ and infinitely many isometries  $\{g_i\} \subset \Gamma$ such that $(g_i\cdot e' )\cap e \neq \varnothing$.
    Let $(U_i,z_i)$ be a small coordinate disk around each puncture $p_i$. 
    First, we consider the case that $e$ and $e'$ have no common endpoints.
    Since the $|\varphi|$-length of $e'\cap (X\setminus \bigcup_i U_i)$ is finite, it contradicts the assumption. 
    Next, we suppose that $e$ and $e'$ have a common endpoint $p_i$.
    Let $\gamma\colon [0,\infty)\to e$ be a smooth curve such that $\gamma(t)$ tends to $p_i$ as $t\to \infty$.
    Then, in the coordinate disk $(U_i,z_i)$, the argument of the tangent vector $\gamma'(t)$ converges as $t\to\infty$.
    It also holds for $e'$ (note that $e'$ is a non-spiraling ideal arc).
    Therefore, in a sufficiently small neighborhood of the puncture, $e$ and $e'$ have no intersection.
    For the same reason as the case of having no common endpoints, this contradicts the assumption.
\end{proof}

From \Cref{proposition:FiniteCopies}, it is sufficient to show the following.

\begin{proposition}
    The distance functions $d_t$ uniformly converge to $I_\varphi$ on a domain covered by the image of $F_\varphi$ by finite elements of $\Gamma$.
\end{proposition}

\begin{proof}
    For simplicity, we show the uniform convergence of the distance functions on a single fundamental domain $F_\varphi$.
    The method of the proof is almost the same as \Cref{theorem:IdealPolygonsCase}.
    However, for the sake of convenience for readers, we see each step one by one.

    The following lemma is shown in the same manner as \Cref{lemma:IntersectionNumbersBoundedFromDistance}
    \begin{lemma}
        For each $t>0$, the inequality $I_\varphi<d_t$ holds on $F_{\varphi}$.
    \end{lemma}

    Next, we estimate infinite vertical or horizontal arcs of $\varphi$ around each puncture.

    \begin{lemma}\label{lemma:GeneralCaseestimateForInfiniteLines}
        Let $p_i$ be a puncture of $X$ and $U_i$ be a sufficiently small sink neighborhood of $p_i$ with respect to $\varphi$.
        We suppose that the $|\varphi|$-length of any nontrivial loop around $p_i$ in $U_i\setminus \{p_i\}$ is at least $2C$ and $\partial U_i$ is of distance at least $C$ away from zeros of $\varphi$ in the $|\varphi|$-metric.
        Then, there exists a sufficiently large $t_0>0$ such that, for every $t>t_0$, there exists a constant $D$ depending only on $\varphi$ such that the following hold:
        \begin{enumerate}
            \item Let $H\colon [0,\infty)\to \widetilde{X}$ be a horizontal line of $\widetilde{\varphi}$ (see \Cref{lemma:estimateForInfintiteLines}).
            We suppose that $H$ is parametrized by the $|\widetilde{\varphi}|$-arc length and the image of $\pi\circ H$ is contained in $U_i$.
            Then, for each $a\geq 0$,
            \begin{equation}\label{eq:HorizontalInfiniteSegment2}
                \ell_{g_t}(H|_{[0,a]})-i(\mathcal{F}_{\widetilde{\varphi},\mathrm{vert}},H|_{[0,a]}) <De^{-C\sqrt{t}}.
            \end{equation}
            \item Let $V\colon [0,\infty)\to \widetilde{X}$ be a vertical line of $\widetilde{\varphi}$ (see \Cref{lemma:estimateForInfintiteLines}).
            We suppose that $V$ is parametrized by the $|\widetilde{\varphi}|$-arc length and $\pi\circ V$ is contained in $U_i$.
            Then, 
            \begin{equation}\label{eq:VerticalInfiniteSegment2}
                \ell_{g_t}(V) < De^{-C\sqrt{t}}.
            \end{equation}
        \end{enumerate}
    \end{lemma}
    \begin{proof}
        If $p_i$ is a higher order pole than $2$, the proof of \Cref{lemma:estimateForInfintiteLines} works here.
        Therefore, we only deal with the case that $p_i$ is a double pole of $\varphi$.
        In this case, we can take a neighborhood $U_i$ of $p_i$ so that $U_i\setminus \{p_i\}$ is isometric to a one-sided infinite Euclidean cylinder whose core length is $2C$.
        Let $\theta\in (0,\pi)$ be the angle between $H$ (or $V$) and $\pi^{-1}(\partial U_i)$ (\Cref{figure:AngleFromCore}).
        \begin{figure}
           \centering
           \begin{overpic}[scale=1.1]{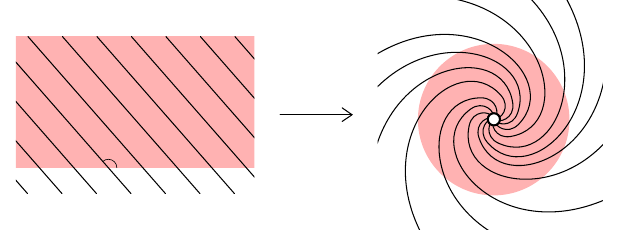}
               \put(49,19){$\pi$}
               \put(42,9){$\pi^{-1}(\partial U_i)$}
               \put(17,12){$\theta$}
               \put(85,4.5){$U_i$}
           \end{overpic}
           \caption{The solid lines denote the horizontal or vertical leaves}
           \label{figure:AngleFromCore}
        \end{figure}
        Then, the distance from $H(s)$ (or $V(s)$) to $\widetilde{\varphi}^{-1}(0)$ is at least $C+s\sin\theta$.
        Therefore, setting $B=\sinh^{-1}(\pi(4g-4+2n+\sum_im_i)/2\pi C^2)$, we have
        \begin{align*}
            \ell_{g_t}(V) &= \frac{1}{\sqrt{2}}\int_V \sqrt{\cosh G(t)-1} \, d\eta \\
            &\leq \frac{1}{\sqrt{2}}\int_V \sqrt{\cosh (Be^{-f\sqrt{t}})-1} \, d\eta & (\text{\Cref{proposition:estimateFromMinsky'sOne}}) \\
            &\leq \frac{A}{\sqrt{2}}\int_V e^{-f\sqrt{t}}\, d\eta & \text{\Cref{eq:BasicInequality}},
        \end{align*}
        where the final inequality holds if $t>(x_0/C)^2$, since $f\circ V(s)>C$.
        We now have 
        \begin{align*}
            \int_V e^{-f\sqrt{t}} d\eta &= \int_0^\infty e^{-f\circ V(s) \sqrt{t}}\,ds \\
            &\leq \int_0^\infty e^{-(C+s\sin\theta)\sqrt{t}} \,ds =\frac{e^{-C\sqrt{t}}}{\sqrt{t}\sin\theta}.
        \end{align*}
        In the same way as the proof of \Cref{lemma:estimateForInfintiteLines}, we can estimate the length of a horizontal line $H$.
        Therefore, we get the conclusion.
    \end{proof}
    By the same manner as \Cref{lemma:DiameterOfEpsilonNeighborhoodOfzero}, we have the following.
    \begin{proposition}\label{lemma:DiameterOfEpsilonNeighborhoodOfzero2}
        Let $p\in \widetilde{X}$ be a zero of $\widetilde{\varphi}$.
        We take $\varepsilon>0$ so that no zero of $\widetilde{\varphi}$ is contained in the $|\widetilde{\varphi}|$-ball $B_{|\widetilde{\varphi}|}(p,3\varepsilon)$ of radius $3\varepsilon$ centered at $p$.
        Then, for a sufficiently large $t$, the diameter of $B_{|\widetilde{\varphi}|}(p,\varepsilon)$ in the $g_t$-metric is at most $M\varepsilon$, where $M$ is a constant depending only on the degree of the zero $p$.
    \end{proposition}

    Finally, we move to the step of deforming a curve as in the proof of \Cref{theorem:IdealPolygonsCase} (as described in \Cref{fig:StepsOfTheDeformation2}).
    We fix a sufficiently small $0<\varepsilon<C$ so that, for any two zeros $z_i,z_j\in F_\varphi$ of $\widetilde{\varphi}$, the disks $B_{|\widetilde{\varphi}|}(z_i,2\varepsilon)$ and $B_{|\widetilde{\varphi}|}(z_j,2\varepsilon)$ have no intersection.
    We take any two points $x_1,x_2 \in F_\varphi$ and a path $\Gamma_0 \subset F_\varphi$ from $x_1$ to $x_2$.
    We deform $\Gamma_0$ via homotopy relative to endpoints.
    (Step 1) We take the $|\widetilde{\varphi}|$-geodesic representative $\Gamma_1 \subset F_\varphi$ in the homotopy class of $\Gamma_0$ rel.\ endpoints.
    Let $U_i (i=1,\ldots,n)$ be sink neighborhoods as in \Cref{lemma:GeneralCaseestimateForInfiniteLines}. 
    Since $F_\varphi\setminus\pi^{-1}(U_i)$ is compact, the total vertical and horizontal measure of $\Gamma_1 \cap (F_\varphi\setminus\pi^{-1}(U_i))$ is bounded from above by the constant $K$ independent of the choice of $x_1,x_2\in F_\varphi$.
    (Step 2) We deform $\Gamma_1$ to a $\varphi$-staircase curve $\Gamma_2$ outside the $\varepsilon$-neighborhood of zeros of $\widetilde{\varphi}$.
    Note that $\Gamma_2$ may not be contained in $F_\varphi$, however the total horizontal and vertical measures of $\Gamma_2$ are the same as  $\Gamma_1$.
    (Step 3) We drag out the horizontal measure of $\Gamma_2$ in the $\varepsilon$-neighborhood of each zero which $\Gamma_2$ passes through.
    For each dragging out operation, the horizontal measure remains the same, however the vertical measure increases by at most $4\varepsilon$.
    Since the number $N$ of zeros of $\widetilde{\varphi}$ in $F_\varphi$ bounds the number of zeros which $\Gamma_1$ passes through from above, the total of increased vertical measures is at most $4N\varepsilon$.
    Let $\Gamma_3$ be the resulting curve.
    (Step 4) We replace each vertical arc contained in $\varepsilon$-neighborhoods of zeros by the $g_t$-geodesic segment keeping the endpoints of the vertical arc.
    \Cref{lemma:DiameterOfEpsilonNeighborhoodOfzero2} implies the $g_t$-length of the geodesic is bounded by $M\varepsilon$.
    We denote the resulting curve by $\Gamma_{4,t}$.

    Similarly to the case of ideal polygons, $\Gamma_{4,t}$ can be decomposed into horizontal line parts $\Gamma_{\mathrm{hori}}$, vertical line parts $\Gamma_{\mathrm{vert}}$, and $g_t$-geodesics part $\Gamma_t$.
    From \Cref{lemma:GeneralCaseestimateForInfiniteLines}
    we find that the $g_t$-lengths of $\Gamma_{\mathrm{hori}}\cap \bigcup \pi^{-1}(U_i)$ and $\Gamma_{\mathrm{vert}}\cap \bigcup \pi^{-1}(U_i)$ is uniformly estimated.
    Therefore we have,
    \begin{align*}
        \ell_{g_t}(\Gamma_{\mathrm{hori}}) &= \ell_{g_t}(\Gamma_{\mathrm{hori}}\cap (X\setminus \bigcup_i \pi^{-1}(U_i)))+\ell_{g_t}(\Gamma_{\mathrm{hori}}\cap \bigcup \pi^{-1}(U_i)) \\
        &\leq i(\widetilde{\mathcal{F}}_{\varphi,\mathrm{vert}},\Gamma_{\mathrm{hori}}\cap (X\setminus \bigcup_i \pi^{-1}(U_i)))(1+De^{-\varepsilon\sqrt{t}}) \\
        &\quad +i(\widetilde{\mathcal{F}}_{\varphi,\mathrm{vert}},\Gamma_{\mathrm{hori}}\cap \bigcup_i \pi^{-1}(U_i))+De^{-\varepsilon \sqrt{t}} \\
        &\leq i(\widetilde{\mathcal{F}}_{\varphi,\mathrm{vert}},\Gamma_{\mathrm{hori}})+(K+1)De^{-\varepsilon\sqrt{t}} \\
        &=i(\widetilde{\mathcal{F}}_{\varphi,\mathrm{vert}},\Gamma_1)+(K+1)De^{-\varepsilon\sqrt{t}} \\
        &=I_\varphi(x_1,x_2)+(K+1)De^{-\varepsilon\sqrt{t}}.
    \end{align*}
    We also have
    \begin{align*}
        \ell_{g_t}(\Gamma_{\mathrm{vert}}) &=\ell_{g_t}(\Gamma_{\mathrm{vert}}\cap (X\setminus \bigcup \pi^{-1}(U_i)))+\ell_{g_t}(\Gamma_{\mathrm{vert}}\cap \bigcup\pi^{-1}(U_i)) \\
        &\leq De^{-\varepsilon\sqrt{t}}(K+4N\varepsilon)+De^{-C\sqrt{t}} \\
        &\leq De^{-\varepsilon\sqrt{t}}(K+4N\varepsilon+1).
    \end{align*}
    And, we have
    \begin{equation*}
        \ell_{g_t}(\Gamma_t)<NM\varepsilon.
    \end{equation*}
    Therefore, we obtain the desired conclusion.
\end{proof}

\begin{remark}
    In \Cref{lemma:estimateForInfintiteLines,lemma:estimateForCompactArcs,lemma:GeneralCaseestimateForInfiniteLines},
    we evaluated the asymptotic length of a horizontal arc $H$ or a vertical arc $V$ along a harmonic map ray.
    In fact, we find from \cite[p.470]{wolf1989harmonic} (see also \cite[Proposition 7.6]{pan2022ray}) that $\cosh G(t)$ monotonically decreases on the universal cover $\widetilde{X}$.
    Therefore, it follows that the $g_t$-length of $H$ or $V$ also monotonically decreases.
\end{remark}

\section{Convergence to the point in the Thurston boundary}

\subsection{The Thurston compactification}\label{subsection:TheThurstonBoundary}

We take $n$ open disks $D_1,\ldots,D_n$ in an oriented, closed surface of genus $g$. We suppose that the closures $\overline{D_1},\ldots,\overline{D_n}$ are pairwise disjoint. We remove $D_1,\ldots,D_n$ and $m_i\geq 0$ points on the boundary of each $D_i$ from the surface.
Let $S$ be the resulting surface.
In this section, we introduce the Thurston boundary of the Teichmüller space of the surface $S$ with closed boundaries or crown ends in a similar way to the paper \cite{alessandrini2016horofunction}, which defines the Thurston compactifications of the Teichmüller spaces for surfaces with only closed boundary.

An \textit{arc} on $S$ is the image of a closed interval by an embedding which maps the boundary of the interval to $\partial S$ and also the interior of the interval to $S\setminus \partial S$.
All homotopies of arcs are relative to $\partial S$.
An arc on $S$ is \textit{essential} if it is not homotopic to $\partial S$.

Let $\mathcal{C}$ denote the set of homotopy classes of essential simple closed curves on $S$, and $\mathcal{A}$ denote the set of homotopy classes of arcs on $S$.
We define the \textit{length function} $\ell_\ast\colon \mathcal{T}(S) \to \R_{\geq 0}^{\mathcal{C}\cup\mathcal{A}}$ as follows: for each $[C,f]\in\mathcal{T}(S)$, 
\begin{equation*}
    \ell_\ast([C,f])=(\ell_{[C,f]}([\gamma]))_{[\gamma]\in\mathcal{C}\cup \mathcal{A}},
\end{equation*}
where $\ell_{[C,f]}([\gamma])=\inf_{\gamma\in[\gamma]} \ell_C(f(\gamma))$.

We set $P(\R_{\geq 0}^{\mathcal{C}\cup\mathcal{A}})=(\R_{\geq 0}^{\mathcal{C}\cup\mathcal{A}}\setminus\{0\})/\R_{>0}$.
Let $\pi\colon \R_{\geq 0}^{\mathcal{C}\cup\mathcal{A}}\setminus\{0\}\to P(\R_{\geq 0}^{\mathcal{C}\cup\mathcal{A}})$ be the projection.
We endow $\R_{\geq 0}^{\mathcal{C}\cup \mathcal{A}}$ with the product topology and $P(\R_{\geq 0}^{\mathcal{C}\cup\mathcal{A}})$ with the quotient topology.

\begin{proposition}\label{proposition:LengthFunctionIsInjective}
    $\ell_\ast\colon \mathcal{T}(S)\to \R_{\geq 0}^{\mathcal{C}\cup\mathcal{A}}$ is injective.
\end{proposition}

\begin{proof}
    A \textit{crown} is an annulus with finite points removed from one of its boundaries.
    If we cut a hyperbolic surface $C\in \mathcal{T}(S)$ along disjoint $3g-3+b+2c$ essential simple closed curves, then $C$ is decomposed into pairs of pants and crowns, where $b$ is the number of closed boundaries of $C$ and $c$ is the number of crown ends of $C$.
    Let $\mathcal{K}$ be the system of curves giving the decomposition.
    As well known, the hyperbolic structure of the pair of pants is determined by the boundary lengths. 
    For the crown with $m_i$ boundary cusps, its hyperbolic structure is determined by the lengths of $m_i$ arcs contained in the crown (see \Cref{figure:Curves}).
    \begin{figure}
       \centering
       \begin{overpic}[]{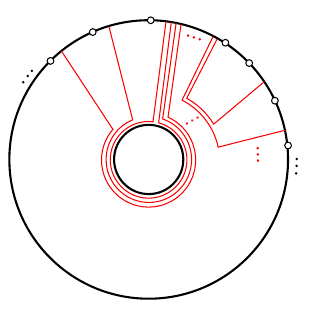}
           \put(46,96){1}
           \put(72,89){2}
           \put(81,81){3}
           \put(89,69){4}
           \put(94,52){5}
           \put(23,93){$m_i$}
           \put(-7,81){$m_i-1$}
       \end{overpic}
       \caption{The lengths of the red $m_i$ curves determine the hyperbolic structure on the crown with $m_i$ boundary cusps.}
       \label{figure:Curves}
    \end{figure}

    Additionally, if we determine the twist parameters associated with each curve in $\mathcal{K}$, the hyperbolic structure of $C$ is completely determined \cite[Lemma 2.16]{gupta2021harmonic}.

    We pick any $[\gamma]\in \mathcal{K}$.
    If $[\gamma]$ is the common boundaries of two (possibly same) pairs of pants, the twist parameter associated with $[\gamma]$ is determined by certain two curves intersecting with $[\gamma]$ (see \cite{fathi2012thurston}).
    We consider the case that $[\gamma]$ is the boundary of a crown.
    To detect the twist parameter along $[\gamma]$, we take any arc $[\alpha]\in \mathcal{A}$ intersecting with $[\gamma]$ (\Cref{figure:HyperbolicStructureOfCrown}).
    \begin{figure}
       \centering
       \begin{overpic}[scale=0.7]{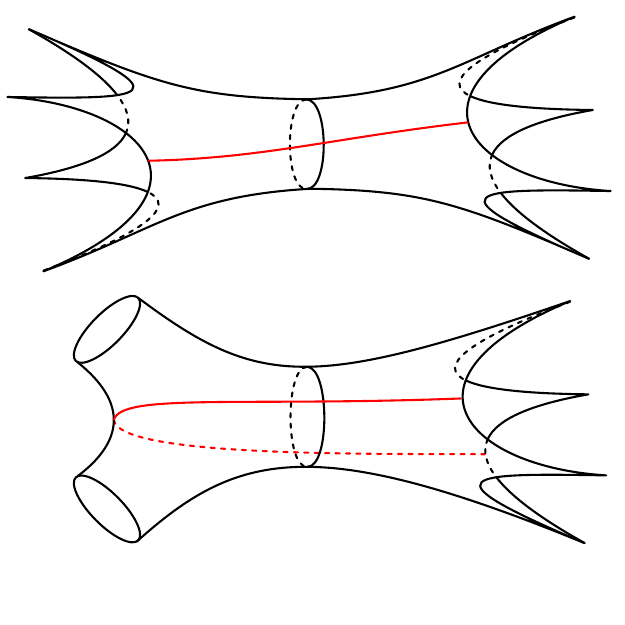}
           \put(47,89){$[\gamma]$}
           \put(47,44){$[\gamma]$}
       \end{overpic}
       \caption{The red curves denote $[\alpha]$.}
       \label{figure:HyperbolicStructureOfCrown}
    \end{figure}
    Let $[\alpha']$ be the arc obtained from $[\alpha]$ by a single Dehn twist along $[\gamma]$.
    Let $C_\theta$ be the hyperbolic surface which is twisted by $\theta\in \R$ along $[\gamma]$.
    By direct calculation on the geometry of hyperbolic plane, the map given by $\theta\mapsto \ell_{C_\theta}([\alpha])$ is strictly convex.
    Therefore, by the same discussion as \cite{fathi2012thurston}, we can detect the twist parameter along $[\gamma]$ by the lengths of $[\alpha]$ and $[\alpha']$.
    Thus, the lengths of $9g-9+b$ essential simple closed curves and $\sum_{i} m_i+2c$ arcs determine the hyperbolic structure of $C$.
\end{proof}
By a similar argument in \cite{alessandrini2016horofunction,liu2010lengthspectra}, the composition $\pi\circ \ell_\ast$ is still injective. Therefore, it follows that $\mathcal{T}(S)$ is embedded into $P(\R_{\geq 0}^{\mathcal{C}\cup\mathcal{A}})$ by $\pi\circ \ell_\ast$.
\begin{definition}
    The closure of the image $\pi\circ \ell_\ast(\mathcal{T}(S))$ is called the \textit{Thurston compactification} of the Teichmüller space of $S$.
    We denote it by $\overline{\mathcal{T}(S)}$.
\end{definition}
Let $S^d$ denote the double of $S$. Then, $S^d$ is a closed surface of genus $2g+n$ with $\sum_{i=1}^n m_i$ punctures.
The Teichmüller space $\mathcal{T}(S)$ is naturally identified with the invariant subset $\mathcal{T}^{\mathrm{sym}}(S^d)$ of the Teichmüller space $\mathcal{T}(S^d)$ of $S^d$ by the involution along the axis of the double.
Since the Thurston compactification $\overline{\mathcal{T}(S^d)}$ of the Teichmüller space $\mathcal{T}(S^d)$ is compact and $\overline{\mathcal{T}^{\mathrm{sym}}(S^d)}$ is a closed subset of $\overline{\mathcal{T}(S^d)}$, we find that $\overline{\mathcal{T}(S)}$ is a compact set.

A \textit{measured foliation} on $S$ is a pair of singular measured foliation on $S$ and transverse measure on it, where a singular measured foliation has singularities as in \Cref{figure:Singularities}.
We define an equivalence relation on the space of measured foliations on $S$ by isotopy and Whitehead moves (see \cite{fathi2012thurston} for more details).
\begin{figure}
   \centering
   \begin{overpic}[scale=0.8]{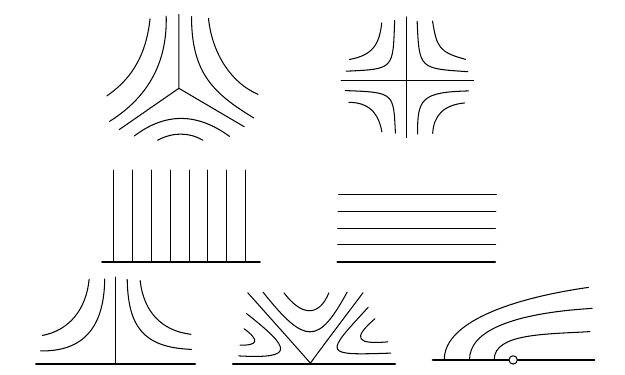}
   \end{overpic}
   \caption{Top: 3 and 4 pronged singularities. Middle: foliation around a boundary. Bottom left and middle: singularities on a boundary. Bottom right: around a removed point on a boundary, there is such a singularity.}
   \label{figure:Singularities}
\end{figure}
Let $\mathcal{MF}(S)$ denote the quotient space of measured foliations on $S$ by the equivalence relation.

Let $X$ be a closed Riemann surface of genus $g$ with $n$ punctures and $\varphi$ be a holomorphic quadratic differential on $X$ with poles of order $m_i+2$ at each puncture $p_i$.
By the way introduced in \Cref{subsection:MeasuredFoliationAndDualR-tree}, we obtain the vertical measured foliation $\mathcal{F}_{\varphi,\mathrm{vert}}$ of $\varphi$ on $X$, which is a \textit{measured foliation with pole singularities} (see \cite{gupta2017meromorphic,gupta2019meromorphic}).
By blowing up each $p_i$, removing points corresponding to asymptotic horizontal direction $\varphi$, and collapsing leaves in half-plane domains of $\varphi$ to removed points, we obtain a measured foliation on $S$ from $\mathcal{F}_{\varphi,\textrm{vert}}$.
The measured foliation on $S$ is denoted by the same symbol $\mathcal{F}_{\varphi,\mathrm{vert}}\in \mathcal{MF}(S)$

Similarly to $\mathcal{T}(S)$, we can identify $\mathcal{MF}(S)$ with the subset $\mathcal{MF}^{\mathrm{sym}}(S^d)\subset \mathcal{MF}(S^d)$ consisting of the invariant measured foliations by the action of the involution of the double.
For a measured foliation $\mathcal{F}=(\mathcal{F},\mu) \in\mathcal{MF}(S)$, we define a function $I_\ast (\mathcal{F})$ on $\mathcal{C}\cup\mathcal{A}$ by, for $[\gamma]\in\mathcal{C}\cup\mathcal{A}$,
\begin{equation*}
    I_\ast (\mathcal{F})([\gamma])=\inf_{\gamma\in[\gamma]}\mu(\gamma).
\end{equation*}
The image of $\mathcal{MF}(S)$ by $\pi\circ I_\ast$ is called the \textit{Thurston boundary}, and it is denoted by $\mathcal{PMF}(S) \subset P(\R_{\geq0}^{\mathcal{C}\cup\mathcal{A}})$.

Since the discussion in \cite{alessandrini2016horofunction} similarly works in our case, we have the following.
\begin{proposition}
    The boundary of $\overline{\mathcal{T}(S)}$ coincides with $\mathcal{PMF}(S)$.
    The Thurston compactification $\overline{\mathcal{T}(S)}$ is homeomorphic to a closed ball of dimension $6g-6+\sum_{i=1}^n(m_i+3)$.
\end{proposition}

\subsection{The limit along a ray in the Thurston boundary}

In this section, we describe the convergence point of the family $\{C_{t\varphi}\}$ in terms of a length function of the isotopy classes of essential simple closed curves and properly embedded, essential arcs.

Let $S$ be the same surface as in \Cref{subsection:TheThurstonBoundary}.
Let $X$ be a closed Riemann surface of genus $g$ with $n$ punctures $p_1,\ldots,p_n$ and let $\varphi$ be a holomorphic quadratic differential on $X$.
As in the previous section, we suppose that $\varphi$ has a pole of order $m_i+2\ (m_i\geq 0)$ at each puncture $p_i$.
We fix a marking $f\colon X\to S\setminus \partial S$ from $(X,\theta(\varphi))$, where $\theta(\varphi)$ is the direction data determined from the principal part of $\varphi$ (see \Cref{section:TeichmullerSpace,subsection:DirectionDataObtainedFromPincipalParts}).
Let $\alpha\colon(-\infty,\infty)\to X$ be a smooth simple arc connecting two punctures $p_i,p_j$ of $X$, namely, $\lim_{t\to\infty}\alpha(t)=p_i, \lim_{t\to-\infty}\alpha(t)=p_j$.
If the unit tangent vectors along $\alpha$ converge to unit tangent vectors at the punctures as $t\to\pm\infty$ and the limits are not in the direction data $\theta(\varphi)$, we say that $\alpha$ has \textit{generic directions}.
Then, the $f$-image of $\alpha$ with its endpoints is a properly embedded arc into $S$ (see \Cref{figure:ArcConnectingBoundary}).
\begin{figure}
   \centering
   \begin{overpic}[scale=0.8]{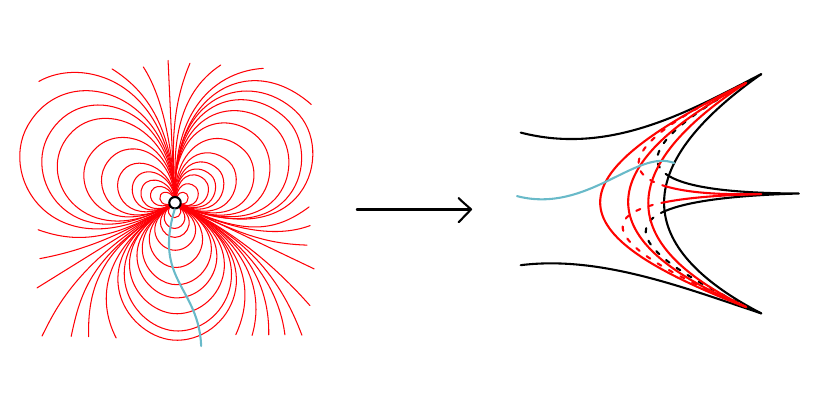}
       \put(49,27){$f$}
   \end{overpic}
   \caption{Since the directions of tangent vectors of the blue arc $\alpha$ converge to a point not in the asymptotic directions of the horizontal measured foliation given by $\varphi$, the arc $\alpha$ is mapped to a properly embedded arc by $f$.}
   \label{figure:ArcConnectingBoundary}
\end{figure}
Let $\alpha_1,\alpha_2\colon(-\infty,\infty)\to X$ be smooth arcs having generic directions.
The arc $\alpha_1$ is \textit{isotopic to} $\alpha_2$ \textit{relative to the direction data} $\theta$, if there exists an isotopy $\{\alpha_t\}$ from $\alpha_1$ to $\alpha_2$ such that the tangent vector of $\alpha_t$ at each of the endpoints always belongs to the same sector between two adjacent directions in $\theta(\varphi)$.
The set $\mathcal{A}$ is naturally identified with the set of isotopy classes (rel.\ the direction data) of essential simple arcs having generic directions.

From the vertical measured foliation $\mathcal{F}_{\varphi,\mathrm{vert}}$, we define a map $I_\ast \mathcal{F_{\varphi,\mathrm{vert}}}\colon \mathcal{C}\cup\mathcal{A}\to \R_{\geq 0}$ by
\begin{equation*}
    I_\ast \mathcal{F}_{\varphi,\mathrm{vert}}([\gamma])=\inf_{\gamma\in[\gamma]} i(\mathcal{F}_{\varphi,\mathrm{vert}},\gamma).
\end{equation*}
The following is a corollary derived from \Cref{theorem:GeneralCases}.
\begin{corollary}\label{corollary:PointInBoundary}
    For each $t>0$, let $[C_{t\varphi},f_{t\varphi}]\in\Teich(S)$ be the hyperbolic surface with $\Psi([C_{t\varphi},f_{t\varphi}])=t\varphi$ (see \Cref{theorem:ParametrizationResult} for the definition of $\Psi$).
    Then,
    \begin{equation*}
        (4t)^{-1/2}\ell_\ast [C_{t\varphi},f_{t\varphi}]\to I_\ast\mathcal{F}_{\varphi,\mathrm{vert}}\in \R^{\mathcal{C}\cup\mathcal{A}}\ \ \ (t\to\infty).
    \end{equation*}   
\end{corollary}

\begin{proof}
    Let $\Gamma$ be a discrete free subgroup of $\operatorname{Isom}^+(\widetilde{X})$ with $X=\widetilde{X}/\Gamma$.
    First, we consider a closed curve $[\gamma]\in\mathcal{C}$.
    Let $g_{[\gamma]}\in \Gamma$ be the holonomy corresponding to the closed curve $[\gamma]$.
    Then the $\frac{1}{4t}C_{t\varphi}$-length of the geodesic representative of $[\gamma]$ is equal to 
    $\inf_{x\in F} d_t(x,g_{[\gamma]}x)$, where $F$ is a fundamental domain of $\Gamma$ which is bounded by lifts of hyperbolic (or Euclidean) geodesic ideal arcs.
    From \Cref{theorem:GeneralCases},
    we find that $d_t(x,g_{[\gamma]}x)$ uniformly converges to $I_\varphi (x,g_{[\gamma]}x)$ on $F$.
    Therefore, the infimum of $d_t(x,g_{[\gamma]}x)$ over $F$ also converges to the infimum of $I_\varphi(x,g_{[\gamma]}x)$ over $F$.
    Since $\inf_{x\in F}I_\varphi(x,g_{[\gamma]}x)=I_\ast\mathcal{F}_{\varphi,\mathrm{vert}}([\gamma])$, the corollary holds for $[\gamma]\in \mathcal{C}$.

    Next, we consider $[\gamma]\in \mathcal{A}$.
    We suppose that isotopy classes of the marking $f_{t\varphi}$ are always the same isotopy class of a marking $f$.
    Let $\partial_1, \partial_2$ be two boundary components connected by $f([\gamma])$.
    For each $i=1,2$,
    if we take a sufficiently small half-plane domain $P_i$ bounded by a horizontal line of $\varphi$ (or a Euclidean cylindrical domain $A_i$) which is corresponding to $\partial_i$, then the harmonic maps $h_{t\varphi}$ map $P_i$ (or $A_i$) to a neighborhood of $\partial_i$  (see \cite{gupta2021harmonic} and \Cref{figure:BehaviorOfMap}).
    \begin{figure}[t]
       \centering
       \begin{overpic}[scale=0.9]{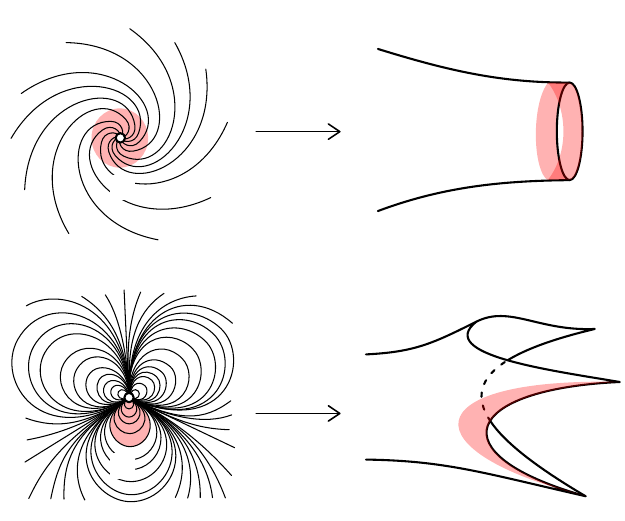}
           \put(18,8){$P_i$}
           \put(18,52){$A_i$}
           \put(45,20){$h_{t\varphi}$}
           \put(45,64.5){$h_{t\varphi}$}
       \end{overpic}
       \caption{The solid curves on the left side are horizontal leaves.}
       \label{figure:BehaviorOfMap}
    \end{figure}
    We can assume that the $|\varphi|$-distance from $P_i$ (or $A_i$) to $\varphi^{-1}(0)$ is at least $\delta>0$.
    For each $t>0$, the $C_{t\varphi}$-geodesic representative $\gamma_t\in [\gamma]$ should pass through $P_i$ (or $A_i$).
    Therefore, we have 
    \begin{equation}\label{eq:Inequality1}
        (4t)^{-1/2} \ell_{C_{t\varphi}} (f_{t\varphi}([\gamma])) \geq \inf _{\substack{x_1\in P_1 (\text{or}A_1), \\ x_2\in P_2 (\text{or}A_2) }} d_t(x_1,x_2).
    \end{equation}
    For each point $x\in P_i$ (or $A_i$), it follows from \Cref{lemma:GeneralCaseestimateForInfiniteLines} that the $g_t$-length of the vertical line starting from $x$ to the puncture is less than $De^{-\delta\sqrt{t}}$ for sufficiently large $t$.
    Therefore, we can connect each point $x\in P_i$ (or $A_i$) and the boundary $\partial_i$ by a curve whose length is less than $De^{-\delta \sqrt{t}}$.
    Thus, we have 
    \begin{equation}\label{eq:Inequality2}
        (4t)^{-1/2} \ell_{C_{t\varphi}} (f_{t\varphi}([\gamma])) \leq \inf _{\substack{x_1\in P_1 (\text{or}A_1), \\ x_2\in P_2 (\text{or}A_2) }} d_t(x_1,x_2)+2De^{-\delta\sqrt{t}}
    \end{equation}
    From \Cref{theorem:GeneralCases}, the distance $d_t(x_1,x_2)$ for two points $x_1\in P_1$ (or $A_1$), $x_2\in P_2$ (or $A_2$) uniformly converges to $I_\varphi(x_1,x_2)$.
    Therefore, from \Cref{eq:Inequality1,eq:Inequality2}, we have  
    \begin{equation*}
        (4t)^{-1/2} \ell_{C_{t\varphi}} (f_{t\varphi}([\gamma]))\to \inf _{\substack{x_1\in P_1 (\text{or}A_1), \\ x_2\in P_2 (\text{or}A_2) }} I_\varphi(x_1,x_2),
    \end{equation*}
    as $t\to\infty$.
    Since $I_\varphi(x_1,x_2)$ realizes the minimal intersection number $I_\ast \mathcal{F}_{\varphi,\mathrm{vert}}([\gamma])$, we obtain the conclusion.
\end{proof}

From \Cref{corollary:PointInBoundary} and the construction of the Thurston compactification, we have the following.

\begin{corollary}
    For each $t>0$, let $[C_{t\varphi},f_{t\varphi}]\in \mathcal{T}(S)$ be the marked hyperbolic surface with $\Psi([C_{t\varphi},f_{t\varphi}])=t\varphi$.
    Then, the hyperbolic surface $[C_{t\varphi},f_{t\varphi}]$ converges to the point $[\mathcal{F}_{\varphi,\mathrm{vert}}] \in \mathcal{PMF}(S)$ in the Thurston boundary as $t\to\infty$.
\end{corollary}

\bibliographystyle{alpha}
\bibliography{DegenerationOfHyperbolicSurfacesWithBoundary}

\end{document}